\documentclass[12pt]{amsart}

\usepackage{amssymb,amsmath,graphicx,color,textcomp, amsthm,bbm,bbold, enumerate,booktabs}
\usepackage{chngcntr}
\usepackage{apptools}
\usepackage{mathrsfs}
\usepackage{tikz}
\usetikzlibrary{trees}

\AtAppendix{\counterwithin{theorem}{section}}
\definecolor{Red}{cmyk}{0,1,1,0}

\definecolor{verde}{cmyk}{1,0,1,0}

\definecolor{loka}{cmyk}{.5,0,1,.5}
\definecolor{azul}{cmyk}{1,1,0,0}

\numberwithin{equation}{section}

\newcommand{\be}{\begin{equation}}
\newcommand{\ee}{\end{equation}}

\newtheorem{theorem}{Theorem}

\newtheorem{definition}{Definition}
\newtheorem{lemma}{Lemma}
\newtheorem{remark}{Remark}
\newtheorem{corollary}[equation]{Corollary}

\begin{document}
\title{Truncated $\mathcal{V}$-fractional Taylor's formula with applications}
\author{J. Vanterler da C. Sousa$^1$}
\address{$^1$ Department of Applied Mathematics, Institute of Mathematics,
 Statistics and Scientific Computation, University of Campinas --
UNICAMP, rua S\'ergio Buarque de Holanda 651,
13083--859, Campinas SP, Brazil\newline
e-mail: {\itshape \texttt{ra160908@ime.unicamp.br, capelas@ime.unicamp.br }}}

\author{E. Capelas de Oliveira$^1$}

\begin{abstract}In this paper, we present and prove a new truncated $\mathcal{V}$-fractional Taylor's formula using the truncated $\mathcal{V}$-fractional variation of constants formula. In this sense, we present the truncated $\mathcal{V}$-fractional Taylor's remainder by means of $\mathcal{V}$-fractional integral, essential for analyzing and comparing the error, when approaching functions by polynomials. From these new results, some applications were made involving some inequalities, specifically, we generalize the Cauchy-Schwartz inequality. 

\vskip.5cm
\noindent
\emph{Keywords}: Truncated $\mathcal{V}$-fractional derivative, $\mathcal{V}$-fractional integral, truncated $\mathcal{V}$-fractional Taylor's remainder, truncated $\mathcal{V}$-fractional Taylor's theorem.
\newline 
MSC 2010 subject classifications. 30K05; 41A58; 26A06; 26A33; 26DXX.
\end{abstract}
\maketitle

\section{Introduction}

The integer-order differential and integral calculus developed by Leibniz and Newton was a great discovery in mathematics. The emergence of new tools and methods for calculating angles, finding solutions to an equation, maximum and minimum of a function, and even making use of a calculator to find the logarithm or exponential of a number, all of these and more can be done due to the Taylor series. Given its importance and relevance in applications, Lagrange realized that the Taylor series was the basic principle of calculus. Thus, over time, many researchers have been interested in studying the Taylor series and its applications, fundamental in various areas of knowledge, such as computation, numerical analysis, engineering, economics, and others \cite{SJBR,RARP,COU}. It is well known that the Taylor series theory is very important to approximate functions by polynomials around any point $x=a$, however, when making certain applications, it is a mistake to replace the function with a given polynomial, although this error is small. Although, this margin of error, its precision and efficiency, make it a powerful tool for applications.

The fractional calculus has proved to be very important and efficient to describe physicals problem and properly for the theoretical advancement in mathematics, physics and other areas. Since then, there are several definitions of fractional derivatives and fractional integrals, of which we mention: Riemann-Liouville, Caputo, Hadamard, Riesz among others, being these based on non-local operators \cite{IP,ECJT,RHM}. However, there are yet some derivatives that fail in the aspect of what is a fractional derivative, such as: the chain rule, the Leibniz rule, and others \cite{JAM,MDJA}.

On the other hand, the concept of a local fractional derivative that has the classical properties of the integer order calculus, has acquired relevance in the scientific community. Recently, Sousa and Oliveira \cite{JEC2} introduced the truncated $\mathcal{V}$-fractional derivative in the domain $\mathbb{R}$, satisfying classical properties of the integer-order calculus, having as special property, to unify five other formulations of local fractional derivatives: conformable fractional, alternative fractional, truncated alternative fractional, $M$-fractional, truncated $M$-fractional \cite{UNT2,KRHA,JEC,JEC1}.

In 2014, Anderson \cite{DRA,DRA1} using local fractional derivatives and integrals called alternative and conformable, proposed an extension of the Taylor's formula in the context of iterated fractional differential equations and discussed some applications. Noting, Anderson's approach is similar to that used for the whole-order derivatives \cite{KWPA}. This paper is devoted to introduces a truncated $\mathcal{V}$-fractional Taylor formula that generalizes these two formulations. An interesting feature of our results is the fact,  in a tentative to approximate functions by polynomials, ensure that the error is smaller in relation to the integer-order derivative and to with the fractional derivative used above.

This paper is organized as follows: in section 2, through the truncated six-parameter Mittag-Leffler function, we present the definition of truncated $\mathcal{V}$-fractional derivative. Also, we present the $\mathcal{V}$-fractional integral and some results derived from both definitions. In section 3, we introduce and prove one of the important results of this article, the truncated $\mathcal{V}$-fractional variations of constant theorem and the truncated $\mathcal{V}$-fractional Taylor's formula. In section 4, to complement the Taylor's formula, we introduce the truncated $\mathcal{V}$-fractional Taylor's remainder by means of $\mathcal{V}$-fractional integral. In section 5, is devoted to $\mathcal{V}$-fractional Hölder inequality and to perform some applications using the Taylor's formula, that is, inequalities. Concluding remarks close the paper.

\section{Preliminaries}

In this section, we will present the definition of the truncated $\mathcal{V}$-fractional derivative through the truncated six parameters Mittag-Leffler function and the gamma function. In this sense, we will present theorems related to continuity and linearity, product, divisibility, as well as the chain rule. We introduce the $\mathcal{V}$-fractional integral of a function $f$. From the definition, we present some theorem about the $\mathcal{V}$-fractional integral; the inverse property, the fundamental theorem of calculus and the integration by parts theorem.

Then, we begin with the definition of the six parameters truncated Mittag-Leffler function given by \cite{JEC2},
\begin{equation}\label{A9}
_{i}\mathbb{E}_{\gamma ,\beta ,p}^{\rho ,\delta ,q}\left( z\right) =\overset{i}{%
\underset{k=0}{\sum }}\frac{\left( \rho \right) _{qk}}{\left( \delta \right)
_{pk}}\frac{z^{k}}{\Gamma \left( \gamma k+\beta \right) },
\end{equation}
being $\gamma ,\beta ,\rho ,\delta \in \mathbb{C}$ and $p,q>0$ such that ${Re}\left( \gamma \right) >0$, ${Re}\left( \beta \right) >0$, ${Re}\left( \rho \right) >0$, ${Re}\left( \delta \right) >0$, ${Re}\left( \gamma \right) +p\geq q$ and $\left( \delta \right) _{pk}$, $\left( \rho \right) _{qk}$ given by
\begin{equation}\label{A6}
\left( \rho \right) _{qk}=\frac{\Gamma \left( \rho +qk\right) }{\Gamma\left( \rho \right) },
\end{equation}
a generalization of the Pochhammer symbol and $\Gamma(\cdot)$ is the function gamma.

From Eq.(\ref{A9}), we introduce the following truncated function, denoted by $_{i}H^{\rho,\delta,q}_{\gamma,\beta,p}(z)$, by means of
\begin{equation}\label{A10}
_{i}H_{\gamma ,\beta ,p}^{\rho ,\delta ,q}\left( z\right) :=\Gamma \left( \beta
\right) \;_{i}\mathbb{E}_{\gamma ,\beta ,p}^{\rho ,\delta ,q}\left( z\right) =\Gamma
\left( \beta \right) \overset{i}{\underset{k=0}{\sum }}\frac{\left( \rho
\right) _{kq}}{\left( \delta \right) _{kp}}\frac{z^{k}}{\Gamma \left( \gamma
k+\beta \right) }.
\end{equation}

In order to simplify notation, in this work, if the truncated $\mathcal{V}$-fractional derivative of order $\alpha$, according to Eq.(\ref{A11}) below, of a function $f$ exists, we simply say that the $f$ function is $\alpha$-differentiable.

So, we start with the following definition, which is a generalization of the usual definition of a derivative presented as a particular limit.

\begin{definition}\label{def7} Let $f:\left[ 0,\infty \right) \rightarrow \mathbb{R}$. For $0<\alpha <1$ the truncated $\mathcal{V}$-fractional derivative of $f$ of order $\alpha$, denoted by $_{i}^{\rho }\mathcal{V}_{\gamma ,\beta ,\alpha }^{\delta ,p,q}(\cdot)$, is defined as
\begin{equation}\label{A11}
_{i}^{\rho }\mathcal{V}_{\gamma ,\beta ,\alpha }^{\delta ,p,q}f\left( t\right) :=%
\underset{\epsilon \rightarrow 0}{\lim }\frac{f\left( t\;_{i}H_{\gamma ,\beta
,p}^{\rho ,\delta ,q}\left( \epsilon t^{-\alpha }\right) \right) -f\left(
t\right) }{\epsilon },
\end{equation}
for all $t>0$, $_{i}H_{\gamma ,\beta ,p}^{\rho ,\delta ,q}\left( \cdot\right) $ is a truncated function  as defined in {\rm Eq.(\ref{A10})} and being $\gamma ,\beta ,\rho ,\delta \in \mathbb{C}$ and $p,q>0$ such that ${Re}\left( \gamma \right) >0$, ${Re}\left( \beta \right) >0$, ${Re}\left( \rho \right) >0$, ${Re}\left( \delta \right) >0$, ${Re}\left( \gamma \right) +p\geq q$ and $\left( \delta \right) _{pk}$, $\left( \rho \right) _{qk}$ given by {\rm Eq.(\ref{A6})} {\rm \cite{JEC2}}. 
\end{definition}

Note that, if $f$ is differentiable in some $(0,a)$, $a>0$ and $\underset{t\rightarrow 0^{+}}{\lim }\, _{i}^{\rho }\mathcal{V}_{\gamma ,\beta ,\alpha }^{\delta ,p,q}f\left( t\right) $ exist, then we have
\begin{equation*}
_{i}^{\rho }\mathcal{V}_{\gamma ,\beta ,\alpha }^{\delta ,p,q}f\left(
0\right) =\underset{t\rightarrow 0^{+}}{\lim }\,_{i}^{\rho }\mathcal{V}_{\gamma ,\beta ,\alpha }^{\delta ,p,q}f\left( t\right).
\end{equation*}

Below, we recover six theorems (the proof can be found in \cite{JEC2}) without proof which are important in what follows.

\begin{theorem}\label{teo1} If the function $f:\left[ 0,\infty \right) \rightarrow \mathbb{R}$ is $\alpha $-differentiable for $t_{0}>0$, with $0<\alpha \leq 1$, then $f$ is continuous in $t_{0}$. 
\end{theorem} 

\begin{theorem}\label{teo2} Let $0<\alpha\leq 1$, $a,b\in\mathbb{R}$, $\gamma ,\beta ,\rho ,\delta \in \mathbb{C}$ and $p,q>0$ such that ${Re}\left( \gamma \right) >0$, ${Re}\left( \beta \right) >0$, ${Re}\left( \rho \right) >0$, ${Re}\left( \delta \right) >0$, ${Re}\left( \gamma \right) +p\geq q$ and $f,g$ $\alpha$-differentiable, for $t>0$. Then,
\begin{enumerate}
\item $_{i}^{\rho }\mathcal{V}_{\gamma ,\beta ,\alpha }^{\delta ,p,q}\left( af+bg\right)
\left( t\right) =a\, _{i}^{\rho }\mathcal{V}_{\gamma ,\beta ,\alpha }^{\delta
,p,q}f\left( t\right) +b\, _{i}^{\rho }\mathcal{V}_{\gamma ,\beta ,\alpha
}^{\delta ,p,q}g\left( t\right) $

\item $_{i}^{\rho }\mathcal{V}_{\gamma ,\beta ,\alpha }^{\delta ,p,q}\left( f\cdot g\right)
\left( t\right) =f\left( t\right) \, _{i}^{\rho }\mathcal{V}_{\gamma ,\beta
,\alpha }^{\delta ,p,q}g\left( t\right)+g\left( t\right) \,
_{i}^{\rho }\mathcal{V}_{\gamma ,\beta ,\alpha }^{\delta ,p,q}f\left( t\right) 
$

\item $_{i}^{\rho }\mathcal{V}_{\gamma ,\beta ,\alpha }^{\delta ,p,q}\left( \frac{f}{g}%
\right) \left( t\right) =\displaystyle\frac{g\left( t\right) \,_{i}^{\rho }\mathcal{V}_{\gamma
,\beta ,\alpha }^{\delta ,p,q}f\left( t\right) -f\left( t\right)
\, _{i}^{\rho }\mathcal{V}_{\gamma ,\beta ,\alpha }^{\delta ,p,q}g\left( t\right) }{\left[ g\left( t\right) \right] ^{2}}$

\item $_{i}^{\rho }\mathcal{V}_{\gamma ,\beta ,\alpha }^{\delta ,p,q}\left( c\right) =0$, where $f(t)=c$ is a constant.

\item If $f$ is differentiable, then $_{i}^{\rho }\mathcal{V}_{\gamma ,\beta ,\alpha }^{\delta ,p,q}f\left( t\right) =\displaystyle\frac{t^{1-\alpha }\Gamma \left( \beta \right) \left( \rho \right) _{q}}{\Gamma\left( \gamma +\beta \right) \left( \delta \right) _{p}}\frac{df\left(
t\right) }{dt}$.

\item $_{i}^{\rho }\mathcal{V}_{\gamma ,\beta ,\alpha }^{\delta ,p,q}\left(
t^{a}\right) =\displaystyle\frac{\Gamma \left( \beta \right) \left( \rho \right) _{q}}{%
\Gamma \left( \gamma +\beta \right) \left( \delta \right) _{p}}at^{a-\alpha }.$
\end{enumerate}
\end{theorem}

\begin{theorem}\label{teo3} {\rm\text{(Chain rule)}} Assume $f,g:(0,\infty)\rightarrow \mathbb{R}$  be two $\alpha$-differentiable functions where $0<\alpha\leq 1$. Let $\gamma ,\beta ,\rho ,\delta \in \mathbb{C}$ and $p,q>0$ such that ${Re}\left( \gamma \right) >0$, ${Re}\left( \beta \right) >0$, ${Re}\left( \rho \right) >0$, ${Re}\left( \delta \right) >0$, ${Re}\left( \gamma \right) +p\geq q$ then $\left( f\circ g\right)$ is $\alpha$-differentiable and for all $t>0$, we have
\begin{equation*}
_{i}^{\rho }\mathcal{V}_{\gamma ,\beta ,\alpha }^{\delta ,p,q}\left( f\circ
g\right) \left( t\right) =f^{\prime }\left( g\left( t\right) \right) \,
_{i}^{\rho }\mathcal{V}_{\gamma ,\beta ,\alpha }^{\delta ,p,q}g\left(
t\right),
\end{equation*}
for $f$ differentiable in $g(t)$.
\end{theorem}

We present the definition of $\mathcal{V}$-fractional integral and some important theorems that are important for the development of the article.

\begin{definition}\label{def9} {\rm($\mathcal{V}$-fractional integral)} Let $a\geq 0$ and $t\geq a$. Also, let $f$ be a function defined on $(a,t]$ and $0<\alpha<1$. Then, the $\mathcal{V}$-fractional integral of $f$ of order $\alpha$ is defined by
\begin{equation}\label{A15}
_{a}^{\rho }\mathcal{I}_{\gamma ,\beta ,\alpha }^{\delta ,p,q}f\left( t\right) :=\frac{\Gamma \left( \gamma +\beta \right) \left( \delta \right) _{p}}{\Gamma\left( \beta \right) \left( \rho \right) _{q}}\int_{a}^{t}\frac{f\left(x\right) }{x^{1-\alpha }}dx,
\end{equation}
with $\gamma ,\beta ,\rho ,\delta \in \mathbb{C}$ and $p,q>0$ such that ${Re}\left( \gamma \right) >0$, ${Re}\left( \beta \right) >0$, ${Re}\left( \rho \right) >0$, ${Re}\left( \delta \right) >0$ and ${Re}\left( \gamma \right) +p\geq q$.
\end{definition}

\begin{remark}\label{fe} In order to simplify notation, in this work, the $\mathcal{V}$-fractional integral of order $\alpha$, is denoted by
\begin{equation*}
\frac{\Gamma \left( \gamma +\beta \right) \left( \delta \right) _{p}}{\Gamma
\left( \beta \right) \left( \rho \right) _{q}}\int_{a}^{b}\frac{f\left(
t\right) }{t^{1-\alpha }}dt=\int_{a}^{b}f\left( t\right) d_{\omega }t
\end{equation*}
where, $d_{\omega }t=\displaystyle\frac{\Gamma \left( \gamma +\beta \right) \left( \delta \right)_{p}}{\Gamma \left( \beta \right) \left( \rho \right) _{q}}t^{\alpha -1}dt$.
\end{remark}

\begin{theorem}\label{teo5} {\rm(Reverse)} Let $a\geq 0$, $t\geq a$ and $0<\alpha<1$. Also, let $f$ be a continuous function such that $_{a}^{\rho }\mathcal{I}_{\gamma ,\beta ,\alpha }^{\delta ,p,q}f\left( t\right)$ exist. Then
\begin{equation*}
_{i}^{\rho }\mathcal{V}_{\gamma ,\beta ,\alpha }^{\delta ,p,q}\left(
_{a}^{\rho }\mathcal{I}_{\gamma ,\beta ,\alpha }^{\delta ,p,q}f\left(
t\right) \right) =f\left( t\right) ,
\end{equation*}
with $\gamma ,\beta ,\rho ,\delta \in \mathbb{C}$ and $p,q>0$ such that ${Re}\left( \gamma \right) >0$, ${Re}\left( \beta \right) >0$, ${Re}\left( \rho \right) >0$, ${Re}\left( \delta \right) >0$ and ${Re}\left( \gamma \right) +p\geq q$.
\end{theorem}

\begin{theorem}\label{teo6} {\rm(Fundamental Theorem of Calculus)} Let $f:(a,b)\rightarrow\mathbb{R}$ be a differentiable function and $0<\alpha\leq 1$. Then, we have
\begin{equation}\label{A16}
_{a}^{\rho }\mathcal{I}_{\gamma ,\beta ,\alpha }^{\delta ,p,q}\left(
_{i}^{\rho }\mathcal{V}_{\gamma ,\beta ,\alpha }^{\delta ,p,q}f\left(
t\right) \right) =f\left( t\right) -f\left( a\right), \qquad t>a ,
\end{equation}
with $\gamma ,\beta ,\rho ,\delta \in \mathbb{C}$ and $p,q>0$ such that ${Re}\left( \gamma \right) >0$, ${Re}\left( \beta \right) >0$, ${Re}\left( \rho \right) >0$, ${Re}\left( \delta \right) >0$ and ${Re}\left( \gamma \right) +p\geq q$.
\end{theorem}

\begin{theorem}\label{teo7} Let $\gamma ,\beta ,\rho ,\delta \in \mathbb{C}$ and $p,q>0$ such that ${Re}\left( \gamma \right) >0$, ${Re}\left( \beta \right) >0$, ${Re}\left( \rho \right) >0$, ${Re}\left( \delta \right) >0$ and ${Re}\left( \gamma \right) +p\geq q$ and $f, g:[a,b]\rightarrow\mathbb{R}$ differentiable functions and $0<\alpha<1$. Then, we have
\begin{equation*}
\int_{a}^{b}f\left( x\right) \, _{i}^{\rho }\mathcal{V}_{\gamma ,\beta,\alpha }^{\delta ,p,q}g\left( x\right)  d_{\omega }x=f\left(x\right) g\left( x\right) \mid _{a}^{b}-\int_{a}^{b}g\left( x\right) \,_{i}^{\rho }\mathcal{V}_{\gamma ,\beta ,\alpha }^{\delta ,p,q}f\left(x\right) d_{\omega }x,
\end{equation*}
with $d_{\omega }x=\displaystyle\frac{\Gamma \left( \gamma +\beta \right) \left( \delta
\right) _{p}}{\Gamma \left( \beta \right) \left( \rho \right) _{q}}\displaystyle\frac{dx}{
x^{1-\alpha }}$.
\end{theorem}


\section{Truncated $\mathcal{V}$-fractional Taylor's theorem}

We present the Cauchy's function and truncated $\mathcal{V}$-fractional differential equation. In this sense, we discuss and prove the truncated $\mathcal{V}$-fractional variations of constants theorem and truncated $\mathcal{V}$-fractional Taylor's formula.

Let $r_{j}:[0,\infty)\rightarrow \mathbb{R}$, $1\leq j\leq n$, $n\in\mathbb{N}_{0}$ continuous functions and consider the higher-order linear truncated $\mathcal{V}$-fractional differential equation:  
\begin{equation}\label{F1}
Ly=0,\qquad \text{ where }\qquad Ly=\, _{i}^{\rho }\mathcal{V}_{\gamma ,\beta ,\alpha }^{\delta ,p,q;n}y +\underset{j=1}{\overset{n}{\sum }}r_{j}\, _{i}^{\rho }\mathcal{V}_{\gamma ,\beta ,\alpha }^{\delta ,p,q;n-j}y,
\end{equation}
where $ _{i}^{\rho }\mathcal{V}_{\gamma ,\beta ,\alpha }^{\delta ,p,q;n}y=\,_{i}^{\rho }\mathcal{V}_{\gamma ,\beta ,\alpha }^{\delta ,p,q;n-1}\left( _{i}^{\rho }\mathcal{V}_{\gamma ,\beta ,\alpha }^{\delta ,p,q}y\right)$, $\gamma ,\beta ,\rho ,\delta \in \mathbb{C}$ and $p,q>0$ such that ${Re}\left( \gamma \right) >0$, ${Re}\left( \beta \right) >0$, ${Re}\left( \rho \right) >0$, ${Re}\left( \delta \right) >0$ and ${Re}\left( \gamma \right) +p\geq q$.

Note that, $ _{i}^{\rho }\mathcal{V}_{\gamma ,\beta ,\alpha }^{\delta ,p,q;n}y $ is a continuous function in $[0,\infty)$, because the function $y:[0,\infty)\rightarrow \mathbb{R}$ is a solution of the Eq.(\ref{F1}) in $[0,\infty)$ knowing that $y$ is $n$ times $\alpha$-fractional differentiable in $[0,\infty)$ and satisfies $Ly(t)=0$, $t \in[0,\infty)$. 

Let $f:[0,\infty)\rightarrow \mathbb{R}$ continuous functions and consider the non homogeneous equation
\begin{equation}\label{F2}
_{i}^{\rho }\mathcal{V}_{\gamma ,\beta ,\alpha }^{\delta
,p,q;n}y\left( t\right) +\underset{j=1}{\overset{n}{\sum }}r_{j}\, _{i}^{\rho }\mathcal{V}_{\gamma ,\beta ,\alpha }^{\delta,p,q;n-j}y\left( t\right) =f\left( t\right) .
\end{equation}

\begin{definition} We define the Cauchy's function, $y:\left[ 0,\infty \right) \times \left[ 0,\infty \right)\rightarrow \mathbb{R}$ for the fractional linear equation {\rm Eq.(\ref{F1})} to be, for each fixed $s\in[0,\infty)$, the solution of the initial value problem
\begin{equation}
Ly=0,\text{ }\, _{i}^{\rho }\mathcal{V}_{\gamma ,\beta ,\alpha }^{\delta ,p,q;j}y\left( s,s\right) =0,\text{ }0\leq j\leq n-2,\text{ }\, _{i}^{\rho }\mathcal{V}_{\gamma ,\beta ,\alpha }^{\delta ,p,q;n-1}y\left( s,s\right)=1,
\end{equation}
where $\gamma ,\beta ,\rho ,\delta \in \mathbb{C}$ and $p,q>0$ such that ${Re}\left( \gamma \right) >0$, ${Re}\left( \beta \right) >0$, ${Re}\left( \rho \right) >0$, ${Re}\left( \delta \right) >0$ and ${Re}\left( \gamma \right) +p\geq q$.
\end{definition}

Note that, $ _{i}^{\rho }\mathcal{V}_{\gamma ,\beta ,\alpha }^{\delta ,p,q} y=0$, where
\begin{equation}\label{F3}
y\left( t,s\right) :=\frac{1}{\left( n-1\right) !}\left( \frac{\Gamma \left(\gamma +\beta \right) \left( \delta \right) _{p}}{\Gamma \left( \beta\right) \left( \rho \right) _{q}}\frac{t^{\alpha }-s^{\alpha }}{\alpha }\right) ^{n-1}
\end{equation}
is the so-called Cauchy's function and can be easily verified using item 6 of the {\rm Theorem \ref{teo2}}.

\begin{theorem}\label{teo7} {\rm (Truncated $\mathcal{V}$-fractional variation of constants)} Let $0<\alpha\leq 1$ and $s,t\in[0,\infty)$. If $f$ is a continuous function, then the solution of the initial value problem
\begin{equation}\label{eq1}
Ly=f\left( t\right),\qquad \ _{i}^{\rho }\mathcal{V}_{\gamma ,\beta ,\alpha}^{\delta ,p,q;j}y\left( s\right) =0,\qquad 0\leq j\leq n-1,
\end{equation}
is given by
\begin{equation}\label{F4}
y\left( t\right) =\int_{s}^{t}y\left( t,\tau \right) f\left( \tau \right)d_{\omega }\tau ,
\end{equation}
where $y(t,\tau)$ is the Cauchy's function for the {\rm Eq.$(\ref{F1})$} and with $\gamma ,\beta ,\rho ,\delta \in \mathbb{C}$ and $p,q>0$ such that ${Re}\left( \gamma \right) >0$, ${Re}\left( \beta \right) >0$, ${Re}\left( \rho \right) >0$, ${Re}\left( \delta \right) >0$ and ${Re}\left( \gamma \right) +p\geq q$.
\end{theorem}

\begin{proof} Applying the derivative operator $_{i}^{\rho }\mathcal{V}_{\gamma ,\beta ,\alpha}^{\delta ,p,q;j}y(\cdot)$ on both sides of {\rm Eq.(\ref{F4})} and using the properties of Cauchy's function, we have
\begin{eqnarray}\label{F5}
_{i}^{\rho }\mathcal{V}_{\gamma ,\beta ,\alpha }^{\delta ,p,q;j}y\left(t\right)  &=&_{i}^{\rho }\mathcal{V}_{\gamma ,\beta ,\alpha }^{\delta,p,q;j}\,\int_{s}^{t}y\left( t,\tau \right) f\left( \tau \right)d_{\omega }\tau    \notag \\
&=&\int_{s}^{t}\, _{i}^{\rho }\mathcal{V}_{\gamma ,\beta ,\alpha}^{\delta ,p,q;j} y\left( t,\tau \right) f\left( \tau \right)d_{\omega }\tau +\,_{i}^{\rho }\mathcal{V}_{\gamma ,\beta ,\alpha}^{\delta ,p,q;j-1} y\left( t,t\right) f\left( t\right)   \notag \\
&=&\int_{s}^{t}\,_{i}^{\rho }\mathcal{V}_{\gamma ,\beta ,\alpha}^{\delta ,p,q;j} y\left( t,\tau \right) f\left( \tau \right)d_{\omega }\tau,
\end{eqnarray}
for $0\leq j\leq n-1$.

On the other hand, for $j=n$, we get
\begin{eqnarray}\label{F6}
_{i}^{\rho }\mathcal{V}_{\gamma ,\beta ,\alpha }^{\delta ,p,q;n}y\left(t\right)  &=&\int_{s}^{t}\, _{i}^{\rho }\mathcal{V}_{\gamma ,\beta,\alpha }^{\delta ,p,q;n} y\left( t,\tau \right) f\left( \tau \right)d_{\omega }\tau +\, _{i}^{\rho }\mathcal{V}_{\gamma ,\beta ,\alpha }^{\delta ,p,q;n-1} y\left( t,t\right) f\left( t\right)   \notag \\
&=&\int_{s}^{t}\,_{i}^{\rho }\mathcal{V}_{\gamma ,\beta ,\alpha}^{\delta ,p,q;j} y\left( t,\tau \right) f\left( \tau \right)d_{\omega }\tau +f\left( t\right),
\end{eqnarray}
because $_{i}^{\rho }\mathcal{V}_{\gamma ,\beta ,\alpha }^{\delta,p,q;n-1} y\left( t,t\right) =1$.

So, from {\rm Eq.(\ref{F5})} and {\rm Eq.(\ref{F6})}, we have
\begin{equation*}
 _{i}^{\rho }\mathcal{V}_{\gamma ,\beta ,\alpha }^{\delta
,p,q;j} y\left( t\right) =0,\ \bigskip 0\leq j\leq n-1
\end{equation*}
and
\begin{eqnarray*}
Ly\left( t\right)  &=&\,_{i}^{\rho }\mathcal{V}_{\gamma ,\beta ,\alpha
}^{\delta ,p,q;j} y\left( t\right) +\underset{j=1}{\overset{n}{\sum }}%
r_{j}\, _{i}^{\rho }\mathcal{V}_{\gamma ,\beta ,\alpha }^{\delta
,p,q;n-j} y\left( t\right)   \notag \\
&=&\int_{s}^{t}\, _{i}^{\rho }\mathcal{V}_{\gamma ,\beta ,\alpha
}^{\delta ,p,q;n} y\left( t,\tau \right) f\left( \tau \right)
d_{\omega }\tau +f\left( t\right) +\underset{j=1}{\overset{n}{\sum }}%
\int_{s}^{t}r_{j}\, _{i}^{\rho }\mathcal{V}_{\gamma ,\beta ,\alpha
}^{\delta ,p,q;n-j} y\left( t,\tau \right) f\left( \tau \right)
d_{\omega }\tau   \notag \\
&=&\int_{s}^{t}\left( \, _{i}^{\rho }\mathcal{V}_{\gamma ,\beta ,\alpha
}^{\delta ,p,q;n} y\left( t,\tau \right) +\underset{j=1}{\overset{n}{%
\sum }}r_{j}\, _{i}^{\rho }\mathcal{V}_{\gamma ,\beta ,\alpha }^{\delta
,p,q;n-j} y\left( t,\tau \right) \right) f\left( \tau \right)
d_{\omega }\tau +f\left( t\right)   \notag \\
&=&\int_{s}^{t}Ly\left( t,\tau \right) f\left( \tau \right) d_{\omega }\tau
+f\left( t\right) =f\left( t\right) .
\end{eqnarray*}

Therefore, we conclude $Ly(t)$=f(t).
\end{proof}

\begin{remark} 
\begin{enumerate}

\item Choosing $\rho=\gamma=\beta=\delta=p=q=1$ and applying the limit $i\rightarrow 0$ at {\rm Eq.(\ref{eq1})}, then the {\rm Theorem \ref{teo7}} becomes the parameter variation theorem for the conformable fractional derivative {\rm \cite{DRA1}}.

\item Taking $\rho=\gamma=\beta=\delta=p=q=1$ and applying the limit $i\rightarrow \infty$ at {\rm Eq.(\ref{eq1})}, then the {\rm Theorem \ref{teo7}}  becomes the parameter variation theorem for the alternative fractional derivative {\rm \cite{DRA}}.
\end{enumerate}
\end{remark}

\begin{theorem}\label{teo8} {\rm(Truncated $\mathcal{V}$-fractional Taylor's formula)} Let $\alpha\in(0,1]$, $n\in\mathbb{N}$, $p,q>0$, $\gamma,\beta,\delta,\rho\in\mathbb{C}$, with $Re(\gamma)>0$, $Re(\beta)>0$, $Re(\delta)>0$, $Re(\rho)>0$ and $Re(\gamma+p)>q$. Suppose $f$ is $(n+1)$ times $\alpha$-fractional differentiable in $[0,\infty)$ and $s,t\in[0,\infty)$. So, we have
\begin{eqnarray}\label{F01}
f\left( t\right)  &=&\underset{k=0}{\overset{n}{\sum }}\frac{1}{k!}\left( \frac{\Gamma \left( \gamma +\beta \right) \left( \delta \right) _{p}}{\Gamma\left( \beta \right) \left( \rho \right) _{q}}\frac{t^{\alpha }-s^{\alpha }}{\alpha }\right) ^{k}\, _{i}^{\rho }\mathcal{V}_{\gamma ,\beta ,\alpha
}^{\delta ,p,q;k} f\left( s\right) +  \notag \\
&&\frac{1}{n!}\int_{s}^{t}\left( \frac{\Gamma \left( \gamma +\beta \right)\left( \delta \right) _{p}}{\Gamma \left( \beta \right) \left( \rho \right)_{q}}\frac{t^{\alpha }-\tau^{\alpha }}{\alpha }\right) ^{n}\, _{i}^{\rho }\mathcal{V}_{\gamma ,\beta ,\alpha }^{\delta ,p,q;n+1} f\left( \tau\right) d_{\omega }\tau .
\end{eqnarray}
\end{theorem}

\begin{proof} Consider the following function $g\left( t\right) :=\,_{i}^{\rho }\mathcal{V}_{\gamma ,\beta ,\alpha }^{\delta ,p,q;n+1} f\left( t\right) $. So, $f$ solves the initial value problem
\begin{equation*}
\, _{i}^{\rho }\mathcal{V}_{\gamma ,\beta ,\alpha }^{\delta,p,q;n+1} x = g,\qquad \, _{i}^{\rho }\mathcal{V}_{\gamma ,\beta ,\alpha }^{\delta ,p,q;k} x\left( s\right) =\,_{i}^{\rho }\mathcal{V}_{\gamma ,\beta ,\alpha }^{\delta ,p,q;k} f\left( s\right),\qquad 0\leq k\leq n,
\end{equation*}
where the Cauchy's function for $ _{i}^{\rho }\mathcal{V}_{\gamma ,\beta ,\alpha }^{\delta ,p,q;n+1} y=0$ is given by
\begin{equation*}
y\left( t,s\right) =\frac{1}{n!}\left( \frac{\Gamma \left( \gamma +\beta\right) \left( \delta \right) _{p}}{\Gamma \left( \beta \right) \left( \rho\right) _{q}}\frac{t^{\alpha }-s^{\alpha }}{\alpha }\right) ^{n}.
\end{equation*} 

Using the variation of constants formula, i.e, $f(t)=u(t)+y(t)$, with
\begin{equation*}
y\left( t\right) =\int_{s}^{t}y\left( t,s\right) g\left( \tau \right)d_{\omega }\tau ,
\end{equation*}
we have
\begin{equation}
f\left( t\right) =u\left( t\right) +\frac{1}{n!}\int_{s}^{t}\left( \frac{\Gamma \left( \gamma +\beta \right) \left( \delta \right) _{p}}{\Gamma \left( \beta \right) \left( \rho \right) _{q}}\frac{t^{\alpha }-\tau	^{\alpha }}{\alpha }\right) ^{n}g\left( \tau \right) d_{\omega }\tau ,
\end{equation}
where $u$ solves the initial value problem:
\begin{equation}\label{F7}
_{i}^{\rho }\mathcal{V}_{\gamma ,\beta ,\alpha }^{\delta ,p,q;n+1} u=0,\qquad \, _{i}^{\rho }\mathcal{V}_{\gamma ,\beta ,\alpha }^{\delta ,p,q;m} u\left( s\right) = \, _{i}^{\rho }\mathcal{V}_{\gamma ,\beta ,\alpha }^{\delta ,p,q;m} f\left( s\right) \text{, }\qquad 0\leq m\leq n.
\end{equation}

Now, let $u(t)$ the solution of {\rm Eq.(\ref{F7})} given by
\begin{equation*}
u\left( t\right) =\overset{n}{\underset{k=0}{\sum }}\frac{1}{k!}\left( \frac{\Gamma \left( \gamma +\beta \right) \left( \delta \right) _{p}}{\Gamma\left( \beta \right) \left( \rho \right) _{q}}\frac{t^{\alpha }-s^{\alpha }}{\alpha }\right) ^{k}\, _{i}^{\rho }\mathcal{V}_{\gamma ,\beta ,\alpha}^{\delta ,p,q;k} f\left( s\right).
\end{equation*}

Thus, to validate $u(t)$, we consider the following set,
\begin{equation*}
w\left( t\right) :=\overset{n}{\underset{k=0}{\sum }}\frac{1}{k!}\left( \frac{\Gamma \left( \gamma +\beta \right) \left( \delta \right) _{p}}{\Gamma\left( \beta \right) \left( \rho \right) _{q}}\frac{t^{\alpha }-s^{\alpha }}{\alpha }\right) ^{k}\, _{i}^{\rho }\mathcal{V}_{\gamma ,\beta ,\alpha
}^{\delta ,p,q;k} f\left( s\right) .
\end{equation*}

So, $ _{i}^{\rho }\mathcal{V}_{\gamma ,\beta ,\alpha }^{\delta ,p,q;n+1} w=0$ and we have
\begin{equation*}
_{i}^{\rho }\mathcal{V}_{\gamma ,\beta ,\alpha }^{\delta ,p,q;m} w\left( t\right) =\, _{i}^{\rho }\mathcal{V}_{\gamma ,\beta ,\alpha }^{\delta ,p,q;m} \left( \overset{n}{\underset{k=0}{%
\sum }}\frac{1}{k!}\left( \frac{\Gamma \left( \gamma +\beta \right) \left( \delta \right) _{p}}{\Gamma \left( \beta \right) \left( \rho \right) _{q}}\frac{t^{\alpha }-s^{\alpha }}{\alpha }\right) ^{k}\, _{i}^{\rho }\mathcal{V}_{\gamma ,\beta ,\alpha }^{\delta ,p,q;k} f\left( s\right) \right) .
\end{equation*}

In fact, note that for $m=1$, we can write
\begin{eqnarray*}
_{i}^{\rho }\mathcal{V}_{\gamma ,\beta ,\alpha }^{\delta ,p,q;1} w\left( t\right)  &=&\,_{i}^{\rho }\mathcal{V}_{\gamma,\beta ,\alpha }^{\delta ,p,q;1} \left( \overset{n}{\underset{k=0}{\sum }}\frac{1}{k!}\left( \frac{\Gamma \left( \gamma +\beta \right) \left(\delta \right) _{p}}{\Gamma \left( \beta \right) \left( \rho \right) _{q}}\frac{t^{\alpha }-s^{\alpha }}{\alpha }\right) ^{k}\left( _{i}^{\rho }
\mathcal{V}_{\gamma ,\beta ,\alpha }^{\delta ,p,q;k}\right) f\left( s\right)
\right)   \notag \\
&=&\overset{n}{\underset{k=0}{\sum }}\frac{t^{1-\alpha }}{k!}\frac{\Gamma\left( \beta \right) \left( \rho \right) _{q}}{\Gamma \left( \gamma +\beta\right) \left( \delta \right) _{p}}k\left( \frac{\Gamma \left( \gamma +\beta\right) \left( \delta \right) _{p}}{\Gamma \left( \beta \right) \left( \rho
\right) _{q}}\frac{t^{\alpha }-s^{\alpha }}{\alpha }\right) ^{k-1}  \notag \\
&&\frac{\alpha t^{\alpha -1}}{\alpha }\frac{\Gamma \left( \gamma +\beta\right) \left( \delta \right) _{p}}{\Gamma \left( \beta \right) \left( \rho\right) _{q}}\,_{i}^{\rho }\mathcal{V}_{\gamma ,\beta ,\alpha }^{\delta,p,q;k} f\left( s\right)   \notag \\
&=&\overset{n}{\underset{k=0}{\sum }}\frac{k}{k!}\left( \frac{\Gamma \left(\gamma +\beta \right) \left( \delta \right) _{p}}{\Gamma \left( \beta\right) \left( \rho \right) _{q}}\frac{t^{\alpha }-s^{\alpha }}{\alpha }\right) ^{k-1}\, _{i}^{\rho }\mathcal{V}_{\gamma ,\beta ,\alpha}^{\delta ,p,q;k} f\left( s\right), 
\end{eqnarray*}

and for $m=2$, we get
\begin{eqnarray*}
 _{i}^{\rho }\mathcal{V}_{\gamma ,\beta ,\alpha }^{\delta ,p,q;2} w\left( t\right)  &=&\overset{n}{\underset{k=0}{\sum }}\frac{k}{k!}\frac{t^{1-\alpha }\Gamma \left( \beta \right) \left( \rho \right) _{q}}{\Gamma \left( \gamma +\beta \right) \left( \delta \right) _{p}}\left(
k-1\right) \left( \frac{\Gamma \left( \gamma +\beta \right) \left( \delta\right) _{p}}{\Gamma \left( \beta \right) \left( \rho \right) _{q}}\frac{t^{\alpha }-s^{\alpha }}{\alpha }\right) ^{k-2}  \notag \\
&&\frac{\alpha t^{\alpha -1}}{\alpha }\frac{\Gamma \left( \gamma +\beta\right) \left( \delta \right) _{p}}{\Gamma \left( \beta \right) \left( \rho\right) _{q}}\,_{i}^{\rho }\mathcal{V}_{\gamma ,\beta ,\alpha }^{\delta,p,q;k} f\left( s\right)   \notag \\
&=&\overset{n}{\underset{k=0}{\sum }}\frac{k\left( k-1\right) }{k!}\left( \frac{\Gamma \left( \gamma +\beta \right) \left( \delta \right) _{p}}{\Gamma\left( \beta \right) \left( \rho \right) _{q}}\frac{t^{\alpha }-s^{\alpha }}{\alpha }\right) ^{k-2}\,_{i}^{\rho }\mathcal{V}_{\gamma ,\beta ,\alpha
}^{\delta ,p,q;k} f\left( s\right) .
\end{eqnarray*}

For any $m$, we have
\begin{eqnarray*}
_{i}^{\rho }\mathcal{V}_{\gamma ,\beta ,\alpha }^{\delta ,p,q;m} w\left( t\right)  &=&\overset{n}{\underset{k=0}{\sum }}\frac{k\left( k-1\right) \cdots\left( k-\left( m+1\right) \right) }{k!}\frac{t^{1-\alpha }\Gamma \left( \beta \right) \left( \rho \right) _{q}}{\Gamma \left( \gamma +\beta \right) \left( \delta \right) _{p}}\left( \frac{\Gamma\left( \gamma +\beta \right) \left( \delta \right) _{p}}{\Gamma \left( \beta\right) \left( \rho \right) _{q}}\frac{t^{\alpha }-s^{\alpha }}{\alpha }
\right) ^{k-m}  \notag \\
&&\frac{\alpha t^{\alpha -1}}{\alpha }\frac{\Gamma \left( \gamma +\beta\right) \left( \delta \right) _{p}}{\Gamma \left( \beta \right) \left( \rho\right) _{q}}\, _{i}^{\rho }\mathcal{V}_{\gamma ,\beta ,\alpha }^{\delta ,p,q;k} f\left( s\right)   \notag \\
&=&\overset{n}{\underset{k=0}{\sum }}\frac{k\left( k-1\right) \cdots\left(k-\left( m+1\right) \right) }{k!}\left( \frac{\Gamma \left( \gamma +\beta\right) \left( \delta \right) _{p}}{\Gamma \left( \beta \right) \left( \rho\right) _{q}}\frac{t^{\alpha }-s^{\alpha }}{\alpha }\right) ^{k-m}\,
_{i}^{\rho }\mathcal{V}_{\gamma ,\beta ,\alpha }^{\delta ,p,q;k} f\left( s\right)   \notag \\
&=&\overset{n}{\underset{k=m}{\sum }}\frac{1}{\left( k-m\right) !}\left( \frac{\Gamma \left( \gamma +\beta \right) \left( \delta \right) _{p}}{\Gamma\left( \beta \right) \left( \rho \right) _{q}}\frac{t^{\alpha }-s^{\alpha }}{\alpha }\right) ^{k-m}\, _{i}^{\rho }\mathcal{V}_{\gamma ,\beta ,\alpha}^{\delta ,p,q;k} f\left( s\right). 
\end{eqnarray*}

Thus, it follows that
\begin{equation*}
_{i}^{\rho }\mathcal{V}_{\gamma ,\beta ,\alpha }^{\delta ,p,q;m} w\left( s\right) =\overset{n}{\underset{k=m}{\sum }}\frac{1}{\left( k-m\right) !}\left( \frac{\Gamma \left( \gamma +\beta \right) \left(\delta \right) _{p}}{\Gamma \left( \beta \right) \left( \rho \right) _{q}}\frac{s^{\alpha }-s^{\alpha }}{\alpha }\right) ^{k-m}\, _{i}^{\rho }\mathcal{V}_{\gamma ,\beta ,\alpha }^{\delta ,p,q;k} f\left( s\right)=\, _{i}^{\rho }\mathcal{V}_{\gamma ,\beta ,\alpha }^{\delta ,p,q;m} f\left( s\right) ,
\end{equation*}
for $0 \leq m\leq n$.

Consequently, we have that: $w$ solves the {\rm Eq.(\ref{F7})}, and then $u=w$ by the uniqueness, which concludes the proof.
\end{proof}

\begin{remark} 
\begin{enumerate}

\item Choosing $\rho=\gamma=\beta=\delta=p=q=1$ and applying the limit $i\rightarrow 0$ at {\rm Eq.(\ref{F01})}, then {\rm Theorem \ref{teo8}} becomes Taylor's formula for the conformable fractional derivative {\rm \cite{DRA1}}.

\item Taking $\rho=\gamma=\beta=\delta=p=q=1$ and applying the limit $i\rightarrow \infty$ at {\rm Eq.(\ref{F01})}, then {\rm Theorem \ref{teo8}} becomes Taylor's formula for the alternative fractional derivative {\rm \cite{DRA}}.
\end{enumerate}
\end{remark}

\section{Truncated $\mathcal{V}$-fractional Taylor's remainder}

In this section, we present the truncated $\mathcal{V}$-fractional remainder function, as well as discussing and proving the main result of the truncated $\mathcal{V}$-fractional Taylor's remainder theorem.

Using the truncated  $\mathcal{V}$-fractional Taylor's formula, we define the  truncated $\mathcal{V}$-fractional remainder function, given by
\begin{equation*}
R_{-1,f}\left( \cdot ,s\right) :=f\left( s\right) 
\end{equation*}
and for $n>-1$,
\begin{eqnarray}\label{KP}
R_{n,f}\left( t,s\right)  &:&=f\left( s\right) -\overset{n}{\underset{k=0}{\sum }}\frac{\,_{i}^{\rho }\mathcal{V}_{\gamma ,\beta ,\alpha }^{\delta,p,q;k} f\left( t\right) }{k!}\left( \frac{\Gamma \left( \gamma+\beta \right) \left( \delta \right) _{p}}{\Gamma \left( \beta \right)
\left( \rho \right) _{q}}\frac{s^{\alpha }-t^{\alpha }}{\alpha }\right) ^{k}\notag \\
&=&\frac{1}{n!}\int_{t}^{s}\left( \frac{\Gamma \left( \gamma +\beta \right)\left( \delta \right) _{p}}{\Gamma \left( \beta \right) \left( \rho \right)_{q}}\frac{s^{\alpha }-\tau ^{\alpha }}{\alpha }\right) ^{n}\,_{i}^{\rho }\mathcal{V}_{\gamma ,\beta ,\alpha }^{\delta ,p,q;n+1} f\left( \tau\right) d_{\omega }\tau, \notag \\ 
\end{eqnarray}
where $\alpha\in(0,1]$ and $\gamma, \beta,\delta,\rho\in\mathbb{C}$ such that $Re(\gamma)>0$, $Re(\beta)>0$, $Re(\delta)>0$, $Re(\rho)>0$, $Re(\gamma)+p\geq q$ and $f$ is $(n+1)$ times $\alpha$-differentiable on $[0,	\infty)$.

\begin{lemma}\label{le1} Let $\alpha\in(0,1]$ and $\gamma, \beta,\delta,\rho\in\mathbb{C}$ such that $Re(\gamma)>0$, $Re(\beta)>0$, $Re(\delta)>0$, $Re(\rho)>0$ and $Re(\gamma)+p\geq q$. The following identity  involving the truncated $\mathcal{V}$-fractional Taylor's remainder, holds
\begin{eqnarray}\label{eq2}
&&\int_{a}^{b}\frac{\, _{i}^{\rho }\mathcal{V}_{\gamma ,\beta ,\alpha }^{\delta
,p,q;n+1} f\left( s\right) }{\left( n+1\right) !}\left( \frac{\Gamma
\left( \gamma +\beta \right) \left( \delta \right) _{p}}{\Gamma \left( \beta
\right) \left( \rho \right) _{q}}\frac{t^{\alpha }-s^{\alpha }}{\alpha }%
\right) ^{n+1}d_{\omega }s  \notag \\
&=&\int_{a}^{t}R_{n,f}\left( a,s\right) d_{\omega
}s+\int_{t}^{b}R_{n,f}\left( b,s\right) d_{\omega }s.
\end{eqnarray}
\end{lemma}

\begin{proof}
We will perform the proof by means of induction on $n$. Then for $n=-1$
\begin{eqnarray*}
&&\int_{a}^{b}\frac{\, _{i}^{\rho }\mathcal{V}_{\gamma ,\beta ,\alpha }^{\delta ,p,q;0} f\left( s\right) }{0!}\left( \frac{\Gamma \left( \gamma +\beta \right) \left( \delta \right) _{p}}{\Gamma \left( \beta \right) \left( \rho \right) _{q}}\frac{t^{\alpha }-s^{\alpha }}{\alpha }\right) ^{0}d_{\omega }s  \notag \\
&=&\int_{a}^{t}R_{-1,f}\left( a,s\right) d_{\omega
}s+\int_{t}^{b}R_{-1,f}\left( b,s\right) d_{\omega }s  \notag \\
&=&\int_{a}^{t}f\left( s\right) d_{\omega }s+\int_{t}^{b}f\left( s\right)d_{\omega }s.
\end{eqnarray*}

Assuming that it is true for $n=k-1$,
\begin{eqnarray*}
&&\int_{a}^{b}\frac{\, _{i}^{\rho }\mathcal{V}_{\gamma ,\beta ,\alpha }^{\delta ,p,q;k} f\left( s\right) }{k!}\left( \frac{\Gamma \left( \gamma +\beta \right) \left( \delta \right) _{p}}{\Gamma \left( \beta \right) \left( \rho \right) _{q}}\frac{t^{\alpha }-s^{\alpha }}{\alpha }\right) ^{k}d_{\omega }s  \notag \\
&=&\int_{a}^{t}R_{k-1,f}\left( a,s\right) d_{\omega
}s+\int_{t}^{b}R_{k-1,f}\left( b,s\right) d_{\omega }s.
\end{eqnarray*}

Thus, taking $n=k$ and using integration by parts for the $\mathcal{V}$-fractional integral, we have
\begin{eqnarray*}
&&\int_{a}^{b}\frac{\, _{i}^{\rho }\mathcal{V}_{\gamma ,\beta ,\alpha }^{\delta,p,q;k+1} f\left( s\right) }{\left( k+1\right) !}\left( \frac{\Gamma\left( \gamma +\beta \right) \left( \delta \right) _{p}}{\Gamma \left( \beta\right) \left( \rho \right) _{q}}\frac{t^{\alpha }-s^{\alpha }}{\alpha }\right) ^{k+1}d_{\omega }s  \notag \\
&=&\frac{\, _{i}^{\rho }\mathcal{V}_{\gamma ,\beta ,\alpha }^{\delta
,p,q;k} f\left( s\right) }{\left( k+1\right) !}\left( \frac{\Gamma\left( \gamma +\beta \right) \left( \delta \right) _{p}}{\Gamma \left( \beta\right) \left( \rho \right) _{q}}\frac{t^{\alpha }-s^{\alpha }}{\alpha }\right) ^{k+1}\mid _{a}^{b}-  \notag \\&&\int_{a}^{b}\frac{\, _{i}^{\rho }\mathcal{V}_{\gamma ,\beta ,\alpha }^{\delta,p,q;k} f\left( s\right) \left( k+1\right) }{k\left( k+1\right) !}\left( \frac{\Gamma \left( \gamma +\beta \right) \left( \delta \right) _{p}}{%
\Gamma \left( \beta \right) \left( \rho \right) _{q}}\frac{t^{\alpha}-s^{\alpha }}{\alpha }\right) ^{k}d_{\omega }s  \notag \\
&=&\frac{\, _{i}^{\rho }\mathcal{V}_{\gamma ,\beta ,\alpha }^{\delta
,p,q;k}f\left( b\right) }{\left( k+1\right) !}\left( \frac{\Gamma\left( \gamma +\beta \right) \left( \delta \right) _{p}}{\Gamma \left( \beta\right) \left( \rho \right) _{q}}\frac{t^{\alpha }-b^{\alpha }}{\alpha }\right) ^{k+1}  \notag \\
&&-\frac{\, _{i}^{\rho }\mathcal{V}_{\gamma ,\beta ,\alpha }^{\delta
,p,q;k} f\left( a\right) }{\left( k+1\right) !}\left( \frac{\Gamma\left( \gamma +\beta \right) \left( \delta \right) _{p}}{\Gamma \left( \beta\right) \left( \rho \right) _{q}}\frac{t^{\alpha }-a^{\alpha }}{\alpha }\right) ^{k+1}  \notag \\
&&+\int_{a}^{b}\frac{\,_{i}^{\rho }\mathcal{V}_{\gamma ,\beta ,\alpha }^{\delta,p,q;k} f\left( s\right) }{k!}\left( \frac{\Gamma \left( \gamma+\beta \right) \left( \delta \right) _{p}}{\Gamma \left( \beta \right)\left( \rho \right) _{q}}\frac{t^{\alpha }-s^{\alpha }}{\alpha }\right)^{k}d_{\omega }s.
\end{eqnarray*}

Using the induction assumption, we conclude that
\begin{eqnarray*}
&&\int_{a}^{b}\frac{\, _{i}^{\rho }\mathcal{V}_{\gamma ,\beta ,\alpha }^{\delta,p,q;k+1} f\left( s\right) }{\left( k+1\right) !}\left( \frac{\Gamma\left( \gamma +\beta \right) \left( \delta \right) _{p}}{\Gamma \left( \beta\right) \left( \rho \right) _{q}}\frac{t^{\alpha }-s^{\alpha }}{\alpha }\right) ^{k+1}d_{\omega }s  \notag \\
&=&\int_{a}^{t}\frac{\, _{i}^{\rho }\mathcal{V}_{\gamma ,\beta ,\alpha }^{\delta,p,q;k} f\left( s\right) }{k!}\left( \frac{\Gamma \left( \gamma+\beta \right) \left( \delta \right) _{p}}{\Gamma \left( \beta \right)\left( \rho \right) _{q}}\frac{t^{\alpha }-s^{\alpha }}{\alpha }\right)
^{k}d_{\omega }s+  \notag \\
&&\int_{t}^{b}\frac{\, _{i}^{\rho }\mathcal{V}_{\gamma ,\beta ,\alpha }^{\delta ,p,q;k} f\left( s\right) }{k!}\left( \frac{\Gamma \left( \gamma +\beta \right) \left( \delta \right) _{p}}{\Gamma \left( \beta \right)\left( \rho \right) _{q}}\frac{t^{\alpha }-s^{\alpha }}{\alpha }\right)^{k}d_{\omega }s  \notag \\
&&+\frac{\, _{i}^{\rho }\mathcal{V}_{\gamma ,\beta ,\alpha }^{\delta
,p,q;k}f\left( b\right) }{\left( k+1\right) !}\left( \frac{\Gamma\left( \gamma +\beta \right) \left( \delta \right) _{p}}{\Gamma \left( \beta\right) \left( \rho \right) _{q}}\frac{t^{\alpha }-b^{\alpha }}{\alpha }\right) ^{k+1}  \notag \\
&&-\frac{\, _{i}^{\rho }\mathcal{V}_{\gamma ,\beta ,\alpha }^{\delta
,p,q;k} f\left( a\right) }{\left( k+1\right) !}\left( \frac{\Gamma \left( \gamma +\beta \right) \left( \delta \right) _{p}}{\Gamma \left( \beta\right) \left( \rho \right) _{q}}\frac{t^{\alpha }-a^{\alpha }}{\alpha }\right) ^{k+1}  \notag \\ 
&=&\int_{a}^{t}R_{k-1,f}\left( a,s\right) d_{\omega
}s+\int_{t}^{b}R_{k-1,f}\left( b,s\right) d_{\omega }s  \notag \\
&&+\frac{\, _{i}^{\rho }\mathcal{V}_{\gamma ,\beta ,\alpha }^{\delta
,p,q;k} f\left( b\right) }{\left( k+1\right) !}\left( \frac{\Gamma\left( \gamma +\beta \right) \left( \delta \right) _{p}}{\Gamma \left( \beta\right) \left( \rho \right) _{q}}\frac{t^{\alpha }-b^{\alpha }}{\alpha }%
\right) ^{k+1}  \notag \\
&&-\frac{\, _{i}^{\rho }\mathcal{V}_{\gamma ,\beta ,\alpha }^{\delta
,p,q;k} f\left( a\right) }{\left( k+1\right) !}\left( \frac{\Gamma \left( \gamma +\beta \right) \left( \delta \right) _{p}}{\Gamma \left( \beta \right) \left( \rho \right) _{q}}\frac{t^{\alpha }-a^{\alpha }}{\alpha }\right) ^{k+1}  \notag \\ &=&\int_{a}^{t}\left[ R_{k-1,f}\left( a,s\right) -\frac{\, _{i}^{\rho }\mathcal{V}_{\gamma ,\beta ,\alpha }^{\delta ,p,q;k} f\left( a\right) }{k!}\left( \frac{\Gamma \left( \gamma +\beta \right) \left( \delta \right) _{p}}{\Gamma \left( \beta \right) \left( \rho \right) _{q}}\frac{s^{\alpha }-a^{\alpha }}{\alpha }\right) ^{k}\right] d_{\omega }s  \notag \\
&&+\int_{t}^{b}\left[ R_{k-1,f}\left( b,s\right) -\frac{\, _{i}^{\rho }\mathcal{V}_{\gamma ,\beta ,\alpha }^{\delta ,p,q;k} f\left( b\right) }{k!}\left( \frac{\Gamma \left( \gamma +\beta \right) \left( \delta \right) _{p}}{\Gamma \left( \beta \right) \left( \rho \right) _{q}}\frac{s^{\alpha }-b^{\alpha }}{\alpha }\right) ^{k}\right] d_{\omega }s  \notag \\
&=&\int_{a}^{t}R_{k,f}\left( a,s\right) d_{\omega
}s+\int_{t}^{b}R_{k,f}\left( b,s\right) d_{\omega }s.
\end{eqnarray*}
\end{proof}

\begin{remark} 
\begin{enumerate}

\item Taking $\rho=\gamma=\beta=\delta=p=q=1$ and applying the limit $i\rightarrow 0$ at {\rm Eq.(\ref{eq2})}, then the {\rm Lemma \ref{le1}} becomes {\rm Lemma 3.1} {\rm \cite{DRA1}}.

\item Taking $\rho=\gamma=\beta=\delta=p=q=1$ and applying the limit $i\rightarrow \infty$ at {\rm Eq.(\ref{eq2})}, then the {\rm Lemma \ref{le1}} becomes {\rm Lemma 2} {\rm \cite{DRA}}.
\end{enumerate}
\end{remark}

\begin{corollary} Let $\alpha\in(0,1]$ and $\gamma, \beta,\delta,\rho\in\mathbb{C}$ such that $Re(\gamma)>0$, $Re(\beta)>0$, $Re(\delta)>0$, $Re(\rho)>0$ and $Re(\gamma)+p\geq q$. Then, for $n\geq -1$, we have
\begin{equation}\label{K1}
\int_{a}^{b}\frac{\left( _{i}^{\rho }\mathcal{V}_{\gamma ,\beta ,\alpha }^{\delta,p,q;n+1}\right) f\left( s\right) }{\left( n+1\right) !}\left( \frac{\Gamma\left( \gamma +\beta \right) \left( \delta \right) _{p}}{\Gamma \left( \beta\right) \left( \rho \right) _{q}}\frac{a^{\alpha }-s^{\alpha }}{\alpha }\right) ^{n+1}d_{\omega }s=\int_{a}^{b}R_{n,f}\left( b,s\right) d_{\omega }s
\end{equation}
and
\begin{equation}\label{K2}
\int_{a}^{b}\frac{\left( _{i}^{\rho }\mathcal{V}_{\gamma ,\beta ,\alpha }^{\delta,p,q;n+1}\right) f\left( s\right) }{\left( n+1\right) !}\left( \frac{\Gamma\left( \gamma +\beta \right) \left( \delta \right) _{p}}{\Gamma \left( \beta\right) \left( \rho \right) _{q}}\frac{b^{\alpha }-s^{\alpha }}{\alpha }\right) ^{n+1}d_{\omega }s=\int_{a}^{b}R_{n,f}\left( a,s\right) d_{\omega }s.
\end{equation}
\end{corollary}

\section{Applications}

Using the truncated $\mathcal{V}$-fractional Taylor's formula and truncated $\mathcal{V}$-fractional Taylor's remainder theorem we realize applications. Besides that, we introduce the Hölder's inequality by means of $\mathcal{V}$-fractional integral, that generalizes the Cauchy-Schwartz inequality \cite{CAUCHY}.

We can give the Hölder's inequality in $\mathcal{V}$-fractional integral as follows:

\begin{lemma}\label{lem1} Let $f,g\in C\left[ a,b\right] $, $r,s>1$ with $\dfrac{1}{r}+\dfrac{1}{s}=1$, then
\begin{equation}\label{K4}
\int_{a}^{b}\left\vert f\left( x\right) g\left( x\right) \right\vert d_{\omega }x\leq \left( \int_{a}^{b}\left\vert f\left( x\right) \right\vert ^{r}d_{\omega }x\right) ^{\frac{1}{r}}\left( \int_{a}^{b}\left\vert g\left(x\right) \right\vert ^{s}d_{\omega }x\right) ^{\frac{1}{s}}.
\end{equation}
\end{lemma}

\begin{remark} For $r=s=2$ in {\rm Lemma \ref{lem1}}, we have the Cauchy-Schwartz inequality for $\mathcal{V}$-fractional integral, i.e.,
\begin{equation}\label{K4}
\int_{a}^{b}\left\vert f\left( x\right) g\left( x\right) \right\vert d_{\omega }x\leq \left( \int_{a}^{b}\left\vert f\left( x\right) \right\vert ^{2}d_{\omega }x\right) ^{\frac{1}{2}}\left( \int_{a}^{b}\left\vert g\left(x\right) \right\vert ^{2}d_{\omega }x\right) ^{\frac{1}{2}}.
\end{equation}
\end{remark}

\begin{theorem}\label{teo9} Let $\alpha\in(0,1]$, $f:[a,b]\rightarrow \mathbb{R}$ be an $n+1$ times $\alpha$-fractional differentiable function, $r,s>1$, $\dfrac{1}{r}+\dfrac{1}{s}=1$, and $t\geq x_{0}$, $t,x_{0}\in[a,b]$ and $\gamma, \beta,\delta,\rho\in\mathbb{C}$ such that $Re(\gamma)>0$, $Re(\beta)>0$, $Re(\delta)>0$, $Re(\rho)>0$ and $Re(\gamma)+p\geq q$. Then, the inequality holds
\begin{eqnarray}\label{eq25}
&&\int_{x_{0}}^{t}\left\vert R_{n,f}\left( x_{0},\tau \right) \right\vert\left\vert \,_{i}^{\rho }\mathcal{V}_{\gamma ,\beta ,\alpha }^{\delta ,p,q;n+1} f\left( \tau \right) \right\vert d_{\omega }\tau   \notag \\
&\leq &\left( \frac{\Gamma \left( \gamma +\beta \right) \left( \delta \right) _{ p}}{\Gamma \left( \beta \right) \left( \rho \right) _{q}}\right) ^{n}\frac{\left( t^{\alpha }-x_{0}^{\alpha }\right) ^{n+2/r}}{\alpha ^{n+2/r}2^{\frac{1}{s}}n!\left[ \left( nr+1\right) \left( nr+2\right) \right] ^{\frac{1}{r}}}\left( \int_{x_{0}}^{t}\left\vert \, _{i}^{\rho }\mathcal{V}_{\gamma ,\beta ,\alpha }^{\delta ,p,q;n+1} f\left( \tau \right)\right\vert ^{s}d_{\omega }\tau \right) ^{\frac{2}{s}}.\notag \\
\end{eqnarray}
\end{theorem}

\begin{proof}
Note that, by $\mathcal{V}$-fractional Taylor's remainder {\rm Eq.(\ref{KP})}, we have
\begin{equation*}
R_{n,f}\left( x_{0},t\right) =\frac{1}{n!}\int_{x_{0}}^{t}\left( \frac{%
\Gamma \left( \gamma +\beta \right) \left( \delta \right) _{p}}{\Gamma
\left( \beta \right) \left( \rho \right) _{q}}\frac{t^{\alpha }-\tau
^{\alpha }}{\alpha }\right) ^{n}\,_{i}^{\rho }\mathcal{V}_{\gamma ,\beta ,\alpha
}^{\delta ,p,q;n+1} f\left( \tau \right) d_{\omega }\tau,
\end{equation*}
with $x_{0},t\in \lbrack a,b]$.

Using the Hölder's inequality for $\mathcal{V}$-fractional integral {\rm Eq.(\ref{lem1})}, it follows that
\begin{eqnarray}\label{K5}
\left\vert R_{n,f}\left( x_{0},t\right) \right\vert  &\leq &\left( \frac{%
\Gamma \left( \gamma +\beta \right) \left( \delta \right) _{p}}{\Gamma
\left( \beta \right) \left( \rho \right) _{q}\alpha }\right) ^{n}\frac{1}{n!}%
\int_{x_{0}}^{t}\left( t^{\alpha }-\tau ^{\alpha }\right) ^{n}\left\vert
\,_{i}^{\rho }\mathcal{V}_{\gamma ,\beta ,\alpha }^{\delta ,p,q;n+1}
f\left( \tau \right) \right\vert ^{s}d_{\omega }\tau  \notag \\
&=&\left( \frac{\Gamma \left( \gamma +\beta \right) \left( \delta \right)
_{p}}{\Gamma \left( \beta \right) \left( \rho \right) _{q}\alpha }\right)
^{n}\frac{1}{n!}\frac{\left( t^{\alpha }-x_{0}^{\alpha }\right) ^{n+\frac{1}{%
r}}}{\alpha ^{\frac{1}{r}}\left( nr+1\right) ^{\frac{1}{r}}}\left( A\left(
t\right) \right) ^{\frac{1}{s}},
\end{eqnarray}
where $A\left( t\right) =\displaystyle\int_{x_{0}}^{t}\left\vert \,_{i}^{\rho }V_{\gamma ,\beta ,\alpha }^{\delta ,p,q;n+1}f\left( \tau \right) \right\vert ^{s}d_{\omega }\tau, $ $x_{0}\leq t\leq b$, $A(x_{0})=0$.

Thus, applying the $\alpha$-differentiable operator $_{i}^{\rho }V_{\gamma ,\beta ,\alpha }^{\delta ,p,q;n+1}(\cdot)$ on both sides of $A(t)$, we have
\begin{eqnarray*}
_{i}^{\rho }\mathcal{V}_{\gamma ,\beta ,\alpha }^{\delta ,p,q;n+1}A\left( t\right)  &=&\left( _{i}^{\rho }\mathcal{V}_{\gamma ,\beta ,\alpha }^{\delta ,p,q;n+1}\right) \left( \int_{x_{0}}^{t}\left\vert \, _{i}^{\rho}\mathcal{V}_{\gamma ,\beta ,\alpha }^{\delta ,p,q;n+1}f\left( \tau \right)
\right\vert ^{s}d_{\omega }\tau \right)   \nonumber \\
&=&\left\vert \, _{i}^{\rho }\mathcal{V}_{\gamma ,\beta ,\alpha }^{\delta,p,q;n+1} f\left( t \right) \right\vert ^{s}
\end{eqnarray*}
and
\begin{equation}\label{K6}
\left\vert _{i}^{\rho }\mathcal{V}_{\gamma ,\beta ,\alpha }^{\delta ,p,q;n+1} f\left( t\right) \right\vert =\,\left( _{i}^{\rho }\mathcal{V}_{\gamma ,\beta ,\alpha }^{\delta ,p,q}A\left( t\right) \right) ^{\frac{1}{s}}.
\end{equation}

Using {\rm Eq.(\ref{K5})} and {\rm Eq.(\ref{K6})}, we get
\begin{eqnarray}\label{K7}
&&\left\vert R_{n,f}\left( x_{0},t\right) \right\vert \left\vert \,_{i}^{\rho }\mathcal{V}_{\gamma ,\beta ,\alpha }^{\delta ,p,q;n+1} f\left( t\right) \right\vert   \nonumber \\
&\leq &\left( \frac{\Gamma \left( \gamma +\beta \right) \left( \delta \right) _{p}}{\Gamma \left( \beta \right) \left( \rho \right) _{q}}\right) ^{n}\frac{\left( t^{\alpha }-x_{0}^{\alpha }\right) ^{n+\frac{1}{r}}}{n!\alpha ^{n+\frac{1}{r}}\left( nr+1\right) ^{\frac{1}{r}}}\left( A\left( t\right) \right) ^{\frac{1}{s}}\left\vert \, _{i}^{\rho }\mathcal{V}_{\gamma ,\beta ,\alpha }^{\delta ,p,q;n+1}f\left( t\right) \right\vert  
\nonumber \\
&=&\left( \frac{\Gamma \left( \gamma +\beta \right) \left( \delta \right) _{p}}{\Gamma \left( \beta \right) \left( \rho \right) _{q}}\right) ^{n}\frac{\left( t^{\alpha }-x_{0}^{\alpha }\right) ^{n+\frac{1}{r}}}{n!\alpha ^{n+\frac{1}{r}}\left( nr+1\right) ^{\frac{1}{r}}}\left[ A\left( t\right) \,_{i}^{\rho }\mathcal{V}_{\gamma ,\beta ,\alpha }^{\delta ,p,q} A\left( t\right) \right] ^{\frac{1}{s}}.
\end{eqnarray}

Integrating the inequality in {\rm Eq.(\ref{K7})} and using the Hölder's inequality for $\mathcal{V}$-fractional integral, we have
\begin{eqnarray*}
&&\int_{x_{0}}^{t}\left\vert R_{n,f}\left( x_{0},\tau \right) \right\vert \left\vert \,_{i}^{\rho }\mathcal{V}_{\gamma ,\beta ,\alpha }^{\delta,p,q;n+1} f\left( \tau \right) \right\vert d_{\omega }\tau \nonumber
\\
&\leq &\left( \frac{\Gamma \left( \gamma +\beta \right) \left( \delta\right) _{p}}{\Gamma \left( \beta \right) \left( \rho \right) _{q}}\right) ^{n}\frac{1}{n!\alpha ^{n+\frac{1}{r}}\left( nr+1\right) ^{\frac{1}{r}}} 
\nonumber \\
&&\left( \int_{x_{0}}^{t}\left( \tau ^{\alpha }-x_{0}^{\alpha }\right)^{nr+1}d_{\omega }\tau \right) ^{\frac{1}{r}}\left( \int_{x_{0}}^{t}A\left(\tau \right) \,_{i}^{\rho }\mathcal{V}_{\gamma ,\beta ,\alpha }^{\delta,p,q} A\left( \tau \right) d_{\omega }\tau \right) ^{\frac{1}{s}} 
\nonumber \\
&=&\left( \frac{\Gamma \left( \gamma +\beta \right) \left( \delta \right)_{p}}{\Gamma \left( \beta \right) \left( \rho \right) _{q}}\right) ^{n}\frac{\left( t^{\alpha }-x_{0}^{\alpha }\right) ^{n+\frac{2}{r}}}{n!\alpha ^{n+\frac{2}{r}}\left[ \left( nr+1\right) \left( nr+2\right) \right] ^{\frac{1}{r
}}}  \nonumber \\
&&\left( \int_{x_{0}}^{t}\int_{x_{0}}^{t}\left\vert \left( _{i}^{\rho }\mathcal{V}_{\gamma ,\beta ,\alpha }^{\delta ,p,q;n+1}\right) f\left( \tau \right) \right\vert ^{s}\left\vert \,_{i}^{\rho }\mathcal{V}_{\gamma ,\beta ,\alpha }^{\delta ,p,q;n+1} f\left( \tau \right) \right\vert ^{s}d_{\omega
}\tau d_{\omega }\tau \right) ^{\frac{1}{s}}  \nonumber \\
&=&\left( \frac{\Gamma \left( \gamma +\beta \right) \left( \delta \right) _{p}}{\Gamma \left( \beta \right) \left( \rho \right) _{q}}\right) ^{n}\frac{\left( t^{\alpha }-x_{0}^{\alpha }\right) ^{n+\frac{2}{r}}}{n!\alpha ^{n+\frac{2}{r}}\left[ \left( nr+1\right) \left( nr+2\right) \right] ^{\frac{1}{r}}}\frac{\left[ A\left( t\right) \right] ^{\frac{2}{s}}}{2^\frac{1}{s}},
\end{eqnarray*}
which completes the proof.
\end{proof}

\begin{remark} Taking $\rho=\gamma=\beta=\delta=p=q=1$ and applying the limit $i\rightarrow 0$ at {\rm Eq.(\ref{eq25})}, then {\rm Theorem \ref{teo9}} becomes {\rm Theorem 8} {\rm \cite{SMZB}}.
\end{remark}

\begin{corollary} Assuming the conditions of {\rm Theorem \ref{teo9}} with $r=s=2$, we get
\begin{eqnarray*}
&&\int_{x_{0}}^{t}\left\vert R_{n,f}\left( x_{0},\tau \right) \right\vert
\left\vert \,_{i}^{\rho }\mathcal{V}_{\gamma ,\beta ,\alpha }^{\delta
,p,q;n+1} f\left( \tau \right) \right\vert d_{\omega }\tau   \nonumber
\\
&\leq &\left( \frac{\Gamma \left( \gamma +\beta \right) \left( \delta
\right) _{p}}{\Gamma \left( \beta \right) \left( \rho \right) _{q}}\right)
^{n}\frac{\left( t^{\alpha }-x_{0}^{\alpha }\right) ^{n+1}}{2\alpha ^{n+1}n!%
\sqrt{\left( 2n+1\right) \left( n+1\right) }}\int_{x_{0}}^{t}\left\vert
\,_{i}^{\rho }\mathcal{V}_{\gamma ,\beta ,\alpha }^{\delta ,p,q;n+1}
f\left( \tau \right) \right\vert ^{2}d_{\omega }\tau. 
\end{eqnarray*}
\end{corollary}

\begin{theorem}\label{teo10} Let $\alpha\in(0,1]$, $f:[a,b]\rightarrow \mathbb{R}$ be an $n+1$ times $\alpha$-fractional differentiable function, $r,s>1$, $\frac{1}{r}+\frac{1}{s}=1$, and $t\leq x_{0}$, $t,x_{0}\in[a,b]$ and $\gamma, \beta,\delta,\rho\in\mathbb{C}$ such that $Re(\gamma)>0$, $Re(\beta)>0$, $Re(\delta)>0$, $Re(\rho)>0$ and $Re(\gamma)+p\geq q$. Then, the following inequality holds
\begin{eqnarray}\label{K8}
&&\int_{t}^{x_{0}}\left\vert R_{n,f}\left( x_{0},\tau \right) \right\vert\left\vert \,_{i}^{\rho }\mathcal{V}_{\gamma ,\beta ,\alpha }^{\delta ,p,q;n+1} f\left( \tau \right) \right\vert d_{\omega }\tau   \nonumber
\\
&\leq &\left( \frac{\Gamma \left( \gamma +\beta \right) \left( \delta \right) _{p}}{\Gamma \left( \beta \right) \left( \rho \right) _{q}}\right) ^{n}\frac{\left( x_{0}^{\alpha }-t^{\alpha }\right) ^{n+2/r}}{2^{1/s}\alpha ^{n+2/r}n!\left[ \left( nr+1\right) \left( nr+2\right) \right] ^{1/r}}\left(
\int_{t}^{x_{0}}\left\vert \,_{i}^{\rho }\mathcal{V}_{\gamma ,\beta ,\alpha }^{\delta ,p,q;n+1} f\left( \tau \right) \right\vert ^{s}d_{\omega }\tau \right) ^{2/s} \notag. \\
\end{eqnarray}
\end{theorem}

\begin{proof} Using the truncated $\mathcal{V}$-fractional Taylor remainder {\rm Eq.(\ref{KP})} and Hölder's inequality for $\mathcal{V}$-fractional integral, we have 
\begin{eqnarray}\label{K9}
&&\left\vert R_{n,f}\left( x_{0},t\right) \right\vert \leq \left( \frac{
\Gamma \left( \gamma +\beta \right) \left( \delta \right) _{p}}{\Gamma
\left( \beta \right) \left( \rho \right) _{q}}\right) ^{n}\frac{1}{\alpha
^{n}n!}\left\vert \int_{x_{0}}^{t}\left( t^{\alpha }-\tau ^{\alpha }\right)
^{n}\, _{i}^{\rho }\mathcal{V}_{\gamma ,\beta ,\alpha }^{\delta ,p,q;n+1}
f\left( \tau \right) d_{\omega }\tau \right\vert   \nonumber \\
&\leq &\left( \frac{\Gamma \left( \gamma +\beta \right) \left( \delta
\right) _{p}}{\Gamma \left( \beta \right) \left( \rho \right) _{q}}\right)
^{n}\frac{1}{\alpha ^{n}n!}\left( \int_{t}^{x_{0}}\left( \tau ^{\alpha
}-t^{\alpha }\right) ^{nr}d_{\omega }\tau \right) ^{1/r}\left(
\int_{t}^{x_{0}}\left\vert \,_{i}^{\rho }\mathcal{V}_{\gamma ,\beta ,\alpha
}^{\delta ,p,q;n+1} f\left( \tau \right) \right\vert ^{s}d_{\omega
}\tau \right) ^{1/s}  \nonumber \\
&=&\left( \frac{\Gamma \left( \gamma +\beta \right) \left( \delta \right)
_{p}}{\Gamma \left( \beta \right) \left( \rho \right) _{q}}\right) ^{n}\frac{%
\left( x_{0}^{\alpha }-t^{\alpha }\right) ^{n+1/r}}{\alpha ^{n+1/r}n!\left(
nr+1\right) ^{1/r}}\left[ A\left( t\right) \right] ^{1/s},
\end{eqnarray}
where $A\left( t\right) =\displaystyle\int_{t}^{x_{0}}\left\vert \,_{i}^{\rho }V_{\gamma
,\beta ,\alpha }^{\delta ,p,q;n+1} f\left( \tau \right) \right\vert ^{s}d_{\omega }\tau $, $a\leq t\leq x_{0}$ and $A\left( x_{0}\right) =0$.

Therefore, we can write
\begin{eqnarray*}
_{i}^{\rho }\mathcal{V}_{\gamma ,\beta ,\alpha }^{\delta ,p,q} A\left(
t\right)  &=&\,_{i}^{\rho }\mathcal{V}_{\gamma ,\beta ,\alpha }^{\delta
,p,q} \left( -\int_{x_{0}}^{t}\left\vert \,_{i}^{\rho }\mathcal{V}_{\gamma
,\beta ,\alpha }^{\delta ,p,q;n+1} f\left( \tau \right) \right\vert
^{s}d_{\omega }\tau \right)   \nonumber \\
&=&-\left\vert \, _{i}^{\rho }\mathcal{V}_{\gamma ,\beta ,\alpha }^{\delta
,p,q;n+1} f\left( t\right) \right\vert ^{s}
\end{eqnarray*}
and
\begin{equation}\label{K10}
\left\vert  _{i}^{\rho }\mathcal{V}_{\gamma ,\beta ,\alpha }^{\delta ,p,q;n+1} f\left( t\right) \right\vert =\left( -\, _{i}^{\rho }\mathcal{V}_{\gamma ,\beta ,\alpha }^{\delta ,p,q} A\left( t\right) \right)^{1/s}.
\end{equation}

Using {\rm Eq.(\ref{K9})} and {\rm Eq.(\ref{K10})}, it follows that
\begin{eqnarray}\label{K11}
&&\left\vert R_{n,f}\left( x_{0},t\right) \right\vert \left\vert _{i}^{\rho
}\mathcal{V}_{\gamma ,\beta ,\alpha }^{\delta ,p,q;n+1}f\left( t\right) \right\vert  
\nonumber \\
&\leq &\left( \frac{\Gamma \left( \gamma +\beta \right) \left( \delta
\right) _{p}}{\Gamma \left( \beta \right) \left( \rho \right) _{q}}\right)
^{n}\frac{\left( x_{0}^{\alpha }-t^{\alpha }\right) ^{n+1/r}}{\alpha
^{n+1/r}n!\left( nr+1\right) ^{1/r}}\left[ A\left( t\right) \right]
^{1/s}\left\vert _{i}^{\rho }\mathcal{V}_{\gamma ,\beta ,\alpha }^{\delta
,p,q;n+1}f\left( t\right) \right\vert   \nonumber \\
&=&\left( \frac{\Gamma \left( \gamma +\beta \right) \left( \delta \right)
_{p}}{\Gamma \left( \beta \right) \left( \rho \right) _{q}}\right) ^{n}\frac{%
\left( x_{0}^{\alpha }-t^{\alpha }\right) ^{n+1/r}}{\alpha ^{n+1/r}n!\left(
nr+1\right) ^{1/r}}\left[ -A\left( t\right) \,_{i}^{\rho }\mathcal{V}_{\gamma
,\beta ,\alpha }^{\delta ,p,q} A\left( t\right) \right] ^{1/s}.\notag \\
\end{eqnarray}

Integrating the inequality in {\rm Eq.(\ref{K11})} and using Hölder's inequality for $\mathcal{V}$-fractional integral, we have
\begin{eqnarray*}
&&\int_{t}^{x_{0}}\left\vert R_{n,f}\left( x_{0},\tau \right) \right\vert
\left\vert _{i}^{\rho }\mathcal{V}_{\gamma ,\beta ,\alpha }^{\delta ,p,q;n+1}f\left(
\tau \right) \right\vert d_{\omega }\tau   \nonumber \\
&\leq &\left( \frac{\Gamma \left( \gamma +\beta \right) \left( \delta
\right) _{p}}{\Gamma \left( \beta \right) \left( \rho \right) _{q}}\right)
^{n}\frac{1}{\alpha ^{n+1/r}n!\left( nr+1\right) ^{1/r}}\left(
\int_{t}^{x_{0}}\left( x_{0}^{\alpha }-\tau ^{\alpha }\right)
^{nr+1}d_{\omega }\tau \right) ^{1/r}  \nonumber \\
&&\times \left( \int_{t}^{x_{0}}A\left( \tau \right) \,_{i}^{\rho
}\mathcal{V}_{\gamma ,\beta ,\alpha }^{\delta ,p,q} A\left( \tau \right)
d_{\omega }\tau \right) ^{1/s}  \nonumber \\
&=&\left( \frac{\Gamma \left( \gamma +\beta \right) \left( \delta \right)
_{p}}{\Gamma \left( \beta \right) \left( \rho \right) _{q}}\right) ^{n}\frac{%
\left( x_{0}^{\alpha }-t^{\alpha }\right) ^{n+2/r}}{\alpha ^{n+2/r}n!\left[
\left( nr+1\right) \left( nr+2\right) \right] ^{1/r}}\left(
\int_{t}^{x_{0}}A\left( \tau \right) \,_{i}^{\rho }\mathcal{V}_{\gamma ,\beta
,\alpha }^{\delta ,p,q} A\left( \tau \right) d_{\omega }\tau \right)
^{1/s}  \nonumber \\
&=&\left( \frac{\Gamma \left( \gamma +\beta \right) \left( \delta \right)
_{p}}{\Gamma \left( \beta \right) \left( \rho \right) _{q}}\right) ^{n}\frac{%
\left( x_{0}^{\alpha }-t^{\alpha }\right) ^{n+2/r}}{\alpha ^{n+2/r}n!\left[
\left( nr+1\right) \left( nr+2\right) \right] ^{1/r}}\frac{\left[ A\left(
s\right) \right] ^{2/s}}{2^{1/s}},
\end{eqnarray*}
which, completes the proof.
\end{proof}

\begin{remark} Taking $\rho=\gamma=\beta=\delta=p=q=1$ and applying the limit $i\rightarrow 0$ at {\rm Eq.(\ref{K8})}, then {\rm Theorem \ref{teo10}} becomes {\rm Theorem 9} {\rm \cite{SMZB}}.
\end{remark}

\begin{corollary} Assuming the conditions of {\rm Theorem \ref{teo10}} with $r=s=2$, we get
\begin{eqnarray*}
&&\int_{t}^{x_{0}}\left\vert R_{n,f}\left( x_{0},\tau \right) \right\vert
\left\vert \,_{i}^{\rho }\mathcal{V}_{\gamma ,\beta ,\alpha }^{\delta
,p,q;n+1} f\left( \tau \right) \right\vert d_{\omega }\tau   \nonumber
 \\
&\leq &\left( \frac{\Gamma \left( \gamma +\beta \right) \left( \delta
\right) _{p}}{\Gamma \left( \beta \right) \left( \rho \right) _{q}}\right)
^{n}\frac{\left( x_{0}^{\alpha }-t^{\alpha }\right) ^{n+1}}{2\alpha ^{n+1}n!%
\sqrt{\left( 2n+1\right) \left( n+1\right) }}\int_{t}^{x_{0}}\left\vert
\,_{i}^{\rho }\mathcal{V}_{\gamma ,\beta ,\alpha }^{\delta ,p,q;n+1}
f\left( \tau \right) \right\vert ^{2}d_{\omega }\tau. \notag \\
\end{eqnarray*}
\end{corollary}

\begin{theorem}\label{teo11} Let $\alpha\in(0,1]$, $f:[a,b]\rightarrow \mathbb{R}$ be an $n+1$ times $\alpha$-fractional differentiable function, $r,s>1$, $\frac{1}{r}+\frac{1}{s}=1$, and $t,x_{0}\in[a,b]$ and $\gamma, \beta,\delta,\rho\in\mathbb{C}$ such that $Re(\gamma)>0$, $Re(\beta)>0$, $Re(\delta)>0$, $Re(\rho)>0$ and $Re(\gamma)+p\geq q$. Then, the following inequality holds
\begin{eqnarray*}
&&\left\vert \int_{x_{0}}^{t}\left\vert R_{n,f}\left( x_{0},\tau \right)
\right\vert \left\vert \, _{i}^{\rho }\mathcal{V}_{\gamma ,\beta ,\alpha }^{\delta,p,q;n+1} f\left( \tau \right) \right\vert d_{\omega }\tau
\right\vert   \nonumber   \\
&\leq &\left( \frac{\Gamma \left( \gamma +\beta \right) \left( \delta
\right) _{p}}{\Gamma \left( \beta \right) \left( \rho \right) _{q}}\right)
^{n}\frac{\left\vert t^{\alpha }-x_{0}^{\alpha }\right\vert ^{n+2/r}}{%
2^{1/s}\alpha ^{n+2/r}n!\left[ \left( nr+1\right) \left( nr+2\right) \right]
^{1/r}}\left\vert \int_{x_{0}}^{t}\left\vert \, _{i}^{\rho }\mathcal{V}_{\gamma
,\beta ,\alpha }^{\delta ,p,q;n+1} f\left( \tau \right) \right\vert
^{s}d_{\omega }\tau \right\vert ^{2/s}.
\end{eqnarray*}
\end{theorem}

\begin{proof} Using {\rm Theorem \ref{teo9}} and {\rm Theorem \ref{teo10}}, the result follows.
\end{proof}

\begin{corollary}\label{coro5} Assuming the conditions of {\rm Theorem \ref{teo11}} with $r=s=2$, we get
\begin{eqnarray*}
&&\left\vert \int_{x_{0}}^{t}\left\vert R_{n,f}\left( x_{0},\tau \right)
\right\vert \left\vert \, _{i}^{\rho }\mathcal{V}_{\gamma ,\beta ,\alpha }^{\delta
,p,q;n+1} f\left( \tau \right) \right\vert d_{\omega }\tau
\right\vert   \nonumber  \\
&\leq &\left( \frac{\Gamma \left( \gamma +\beta \right) \left( \delta
\right) _{p}}{\Gamma \left( \beta \right) \left( \rho \right) _{q}}\right)
^{n}\frac{\left\vert t^{\alpha }-x_{0}^{\alpha }\right\vert ^{n+1}}{2\alpha
^{n+1}n!\sqrt{\left( n+1\right) \left( 2n+1\right) }}\left\vert
\int_{x_{0}}^{t}\left\vert \, _{i}^{\rho }\mathcal{V}_{\gamma ,\beta ,\alpha
}^{\delta ,p,q;n+1} f\left( \tau \right) \right\vert ^{2}d_{\omega
}\tau \right\vert .
\end{eqnarray*}
\end{corollary}

\begin{theorem}\label{teo12} Let $\alpha\in(0,1]$, $f:[a,b]\rightarrow \mathbb{R}$ be an $n+1$ times $\alpha$-fractional differentiable function, with $r=1$, $s=\infty$, $t\in[x_{0},b]$ and $\gamma, \beta,\delta,\rho\in\mathbb{C}$ such that $Re(\gamma)>0$, $Re(\beta)>0$, $Re(\delta)>0$, $Re(\rho)>0$ and $Re(\gamma)+p\geq q$. Then, the inequality holds
\begin{eqnarray}\label{eq35}
&&\int_{x_{0}}^{t}\left\vert R_{n,f}\left( x_{0},\tau \right)
\right\vert \left\vert _{i}^{\rho }\mathcal{V}_{\gamma ,\beta ,\alpha }^{\delta
,p,q;n+1}f\left( \tau \right) \right\vert d_{\omega }\tau 
\nonumber \\
&\leq &\left( \frac{\Gamma \left( \gamma +\beta \right) \left( \delta
\right) _{p}}{\Gamma \left( \beta \right) \left( \rho \right) _{q}}\right)
^{n}\frac{\left( t^{\alpha
}-x_{0}^{\alpha }\right) ^{n+2}}{\alpha ^{n+2}\left( n+2\right) !} \left\Vert _{i}^{\rho }\mathcal{V}_{\gamma ,\beta ,\alpha }^{\delta,p,q;n+1}f\right\Vert _{\infty ,\left[ x_{0},b\right] }^{2},
\end{eqnarray}
where
\begin{equation*}
\left\Vert _{i}^{\rho }V_{\gamma ,\beta ,\alpha }^{\delta
,p,q;n+1}f\right\Vert _{\infty}:=\underset{t\in %
\left[ a,b\right] }{\sup }\left\vert _{i}^{\rho }V_{\gamma ,\beta ,\alpha
}^{\delta ,p,q;n+1}f\left( t\right) \right\vert. 
\end{equation*}
\end{theorem}

\begin{proof} Using the truncated $\mathcal{V}$-fractional Taylor remainder {\rm Eq.(\ref{KP})}, we have
\begin{eqnarray}\label{K14}
\left\vert R_{n,f}\left( x_{0},t\right) \right\vert  &\leq &\left( \frac{%
\Gamma \left( \gamma +\beta \right) \left( \delta \right) _{p}}{\Gamma
\left( \beta \right) \left( \rho \right) _{q}}\right) ^{n}\frac{1}{\alpha
^{n}n!}\int_{x_{0}}^{t}\left( t^{\alpha }-\tau ^{\alpha }\right)
^{n}\left\vert _{i}^{\rho }\mathcal{V}_{\gamma ,\beta ,\alpha }^{\delta
,p,q;n+1}f\left( \tau \right) \right\vert d_{\omega }\tau   \nonumber \\
&\leq&\left( \frac{\Gamma \left( \gamma +\beta \right) \left( \delta \right)
_{p}}{\Gamma \left( \beta \right) \left( \rho \right) _{q}}\right) ^{n}\frac{%
\left\Vert _{i}^{\rho }\mathcal{V}_{\gamma ,\beta ,\alpha }^{\delta
,p,q;n+1}f\right\Vert _{\infty ,\left[ x_{0},b\right] }}{\alpha ^{n}n!}%
\int_{x_{0}}^{t}\left( t^{\alpha }-\tau ^{\alpha }\right) ^{n}d_{\omega
}\tau   \nonumber \\
&=&\left( \frac{\Gamma \left( \gamma +\beta \right) \left( \delta \right)
_{p}}{\Gamma \left( \beta \right) \left( \rho \right) _{q}}\right) ^{n}\frac{%
\left\Vert _{i}^{\rho }\mathcal{V}_{\gamma ,\beta ,\alpha }^{\delta
,p,q;n+1}f\right\Vert _{\infty ,\left[ x_{0},b\right] }}{\alpha ^{n+1}\left(
n+1\right) !}\left( t^{\alpha }-x_{0}^{\alpha }\right) ^{n+1}.
\end{eqnarray}

Moreover, as
\begin{equation*}
\left\vert _{i}^{\rho }\mathcal{V}_{\gamma ,\beta ,\alpha }^{\delta ,p,q;n+1}f\left(
t\right) \right\vert \leq \left\Vert _{i}^{\rho }\mathcal{V}_{\gamma ,\beta ,\alpha
}^{\delta ,p,q;n+1}f\right\Vert _{\infty ,\left[ x_{0},b\right] },
\end{equation*}
for all $t\in[x_{0},b]$ and multiplying by $\left\vert _{i}^{\rho }\mathcal{V}_{\gamma ,\beta ,\alpha }^{\delta ,p,q;n+1}f\left( t\right) \right\vert $ on both sides of {\rm Eq.(\ref{K14})}, it follows that
\begin{equation}\label{K15}
\left\vert R_{n,f}\left( x_{0},t\right) \right\vert \left\vert _{i}^{\rho
}\mathcal{V}_{\gamma ,\beta ,\alpha }^{\delta ,p,q;n+1}f\left( t\right) \right\vert
\leq \left( \frac{\Gamma \left( \gamma +\beta \right) \left( \delta \right)
_{p}}{\Gamma \left( \beta \right) \left( \rho \right) _{q}}\right) ^{n}\frac{%
\left\Vert _{i}^{\rho }\mathcal{V}_{\gamma ,\beta ,\alpha }^{\delta
,p,q;n+1}f\right\Vert _{\infty ,\left[ x_{0},b\right] }^{2}}{\alpha
^{n+1}\left( n+1\right) !}\left( t^{\alpha }-x_{0}^{\alpha }\right) ^{n+1}.
\end{equation}

Integrating the inequality {\rm Eq.(\ref{K15})}, we have
\begin{eqnarray*}
\int_{x_{0}}^{t}\left\vert R_{n,f}\left( x_{0},\tau\right) \right\vert
\left\vert _{i}^{\rho }\mathcal{V}_{\gamma ,\beta ,\alpha }^{\delta ,p,q;n+1}f\left(
\tau\right) \right\vert d_{\omega }\tau  &\leq &\left( \frac{\Gamma \left(
\gamma +\beta \right) \left( \delta \right) _{p}}{\Gamma \left( \beta
\right) \left( \rho \right) _{q}}\right) ^{n}\frac{\left\Vert _{i}^{\rho
}\mathcal{V}_{\gamma ,\beta ,\alpha }^{\delta ,p,q;n+1}f\right\Vert _{\infty ,\left[
x_{0},b\right] }^{2}}{\alpha ^{n+1}\left( n+1\right) !}  \nonumber \\
&&\int_{x_{0}}^{t}\left( \tau ^{\alpha }-x_{0}^{\alpha }\right)
^{n+1}d_{\omega }\tau   \nonumber \\
&=&\left( \frac{\Gamma \left( \gamma +\beta \right) \left( \delta \right)
_{p}}{\Gamma \left( \beta \right) \left( \rho \right) _{q}}\right) ^{n}\frac{%
\left( t^{\alpha }-x_{0}^{\alpha }\right) ^{n+2}}{\alpha ^{n+2}\left(
n+2\right) !}\left\Vert _{i}^{\rho }\mathcal{V}_{\gamma ,\beta ,\alpha }^{\delta
,p,q;n+1}f\right\Vert _{\infty ,\left[ x_{0},b\right] }^{2},
\end{eqnarray*}
which, completes the proof.
\end{proof}

\begin{remark} Taking $\rho=\gamma=\beta=\delta=p=q=1$ and applying the limit $i\rightarrow 0$ at {\rm Eq.(\ref{eq35})}, then {\rm Theorem \ref{teo12}} becomes {\rm Theorem 11} {\rm \cite{SMZB}}.
\end{remark}

\begin{theorem}\label{teo13} Let $\alpha\in(0,1]$, $f:[a,b]\rightarrow \mathbb{R}$ be an $n+1$ times $\alpha$-fractional differentiable function, with $r=1$, $s=\infty$, $t\in[a,x_{0}]$ and $\gamma, \beta,\delta,\rho\in\mathbb{C}$ such that $Re(\gamma)>0$, $Re(\beta)>0$, $Re(\delta)>0$, $Re(\rho)>0$ and $Re(\gamma)+p\geq q$. Then, the inequality holds
\begin{eqnarray}\label{eq38}
&&\int_{x_{0}}^{t}\left\vert R_{n,f}\left( x_{0},\tau \right)
\right\vert \left\vert _{i}^{\rho }\mathcal{V}_{\gamma ,\beta ,\alpha }^{\delta
,p,q;n+1}f\left( \tau \right) \right\vert d_{\omega }\tau 
\nonumber \\
&\leq &\left( \frac{\Gamma \left( \gamma +\beta \right) \left( \delta
\right) _{p}}{\Gamma \left( \beta \right) \left( \rho \right) _{q}}\right)
^{n}\frac{\left( x_{0}^{\alpha }-t^{\alpha}\right) ^{n+2}}{\alpha ^{n+2}\left( n+2\right) !} \left\Vert _{i}^{\rho }\mathcal{V}_{\gamma ,\beta ,\alpha }^{\delta,p,q;n+1}f\right\Vert _{\infty ,\left[a, x_{0}\right] }^{2},
\end{eqnarray}
where
\begin{equation*}
\left\Vert _{i}^{\rho }\mathcal{V}_{\gamma ,\beta ,\alpha }^{\delta
,p,q;n+1}f\right\Vert _{\infty}:=\underset{t\in %
\left[ a,b\right] }{\sup }\left\vert _{i}^{\rho }\mathcal{V}_{\gamma ,\beta ,\alpha
}^{\delta ,p,q;n+1}f\left( t\right) \right\vert. 
\end{equation*}
\end{theorem}

\begin{proof} 
Using the $\mathcal{V}$-fractional Taylor remainder {\rm Eq.(\ref{KP})},
we have 
\begin{eqnarray}\label{K16}
\left\vert R_{n,f}\left( x_{0},t\right) \right\vert  &\leq &\left( \frac{%
\Gamma \left( \gamma +\beta \right) \left( \delta \right) _{p}}{\Gamma
\left( \beta \right) \left( \rho \right) _{q}}\right) ^{n}\frac{1}{\alpha
^{n}n!}\int_{t}^{x_{0}}\left( \tau ^{\alpha }-t^{\alpha }\right)
^{n}\left\vert _{i}^{\rho }\mathcal{V}_{\gamma ,\beta ,\alpha }^{\delta
,p,q;n+1}f\left( \tau \right) \right\vert d_{\omega }\tau   \nonumber \\
&\leq&\left( \frac{\Gamma \left( \gamma +\beta \right) \left( \delta \right)
_{p}}{\Gamma \left( \beta \right) \left( \rho \right) _{q}}\right) ^{n}\frac{%
\left\Vert _{i}^{\rho }\mathcal{V}_{\gamma ,\beta ,\alpha }^{\delta
,p,q;n+1}f\right\Vert _{\infty ,\left[ a,x_{0}\right] }}{\alpha ^{n}n!}%
\int_{t}^{x_{0}}\left( \tau ^{\alpha }-t^{\alpha }\right) ^{n}d_{\omega
}\tau   \nonumber \\
&=&\left( \frac{\Gamma \left( \gamma +\beta \right) \left( \delta \right)
_{p}}{\Gamma \left( \beta \right) \left( \rho \right) _{q}}\right) ^{n}\frac{%
\left\Vert _{i}^{\rho }\mathcal{V}_{\gamma ,\beta ,\alpha }^{\delta
,p,q;n+1}f\right\Vert _{\infty ,\left[ a,x_{0}\right] }}{\alpha ^{n+1}\left(
n+1\right) !}\left( x_{0}^{\alpha }-t^{\alpha }\right) ^{n+1}.
\end{eqnarray}

Moreover, as \[\left\vert _{i}^{\rho }\mathcal{V}_{\gamma ,\beta ,\alpha }^{\delta ,p,q;n+1}f\left(
t\right) \right\vert \leq \left\Vert _{i}^{\rho }\mathcal{V}_{\gamma ,\beta ,\alpha
}^{\delta ,p,q;n+1}f\right\Vert _{\infty ,\left[ a,x_{0}\right] },
\] for all $t\in \left[ a,x_{0}\right] $ and multiplying by $\left\vert
_{i}^{\rho }\mathcal{V}_{\gamma ,\beta ,\alpha }^{\delta ,p,q;n+1}f\left( t\right)
\right\vert $ on both sides of {\rm Eq.(\ref{K16})}, it follows that 
\begin{equation}\label{K17}
\left\vert R_{n,f}\left( x_{0},t\right) \right\vert \left\vert _{i}^{\rho
}\mathcal{V}_{\gamma ,\beta ,\alpha }^{\delta ,p,q;n+1}f\left( t\right) \right\vert
\leq \left( \frac{\Gamma \left( \gamma +\beta \right) \left( \delta \right)
_{p}}{\Gamma \left( \beta \right) \left( \rho \right) _{q}}\right) ^{n}\frac{%
\left\Vert _{i}^{\rho }\mathcal{V}_{\gamma ,\beta ,\alpha }^{\delta
,p,q;n+1}f\right\Vert _{\infty ,\left[a, x_{0}\right] }^{2}}{\alpha
^{n+1}\left( n+1\right) !}\left( t^{\alpha }-x_{0}^{\alpha }\right) ^{n+1}.
\end{equation}

Integrating the inequality in {\rm Eq.(\ref{K17})}, we have 
\begin{eqnarray*}
\int_{x_{0}}^{t}\left\vert R_{n,f}\left( x_{0},\tau \right) \right\vert
\left\vert _{i}^{\rho }\mathcal{V}_{\gamma ,\beta ,\alpha }^{\delta ,p,q;n+1}f\left(
\tau \right) \right\vert d_{\omega }\tau  &\leq &\left( \frac{\Gamma \left(
\gamma +\beta \right) \left( \delta \right) _{p}}{\Gamma \left( \beta
\right) \left( \rho \right) _{q}}\right) ^{n}\frac{\left\Vert _{i}^{\rho
}\mathcal{V}_{\gamma ,\beta ,\alpha }^{\delta ,p,q;n+1}f\right\Vert _{\infty ,\left[
a,x_{0}\right] }^{2}}{\alpha ^{n+1}\left( n+1\right) !} \\
&&\int_{x_{0}}^{t}\left( x_{0}^{\alpha }-\tau ^{\alpha }\right)
^{n+1}d_{\omega }\tau  \\
&=&\left( \frac{\Gamma \left( \gamma +\beta \right) \left( \delta \right)
_{p}}{\Gamma \left( \beta \right) \left( \rho \right) _{q}}\right) ^{n}\frac{
\left( x_{0}^{\alpha }-t^{\alpha }\right) ^{n+2}}{\alpha ^{n+2}\left(
n+2\right) !}\left\Vert _{i}^{\rho }\mathcal{V}_{\gamma ,\beta ,\alpha }^{\delta
,p,q;n+1}f\right\Vert _{\infty ,\left[ a,x_{0}\right] }^{2},
\end{eqnarray*}
which, completes the proof.
\end{proof}

\begin{remark} For $\rho=\gamma=\beta=\delta=p=q=1$ and applying the limit $i\rightarrow 0$ at {\rm Eq.(\ref{eq38})}, then {\rm Theorem \ref{teo13}} becomes {\rm Theorem 12} {\rm \cite{SMZB}}.
\end{remark}

Finally,the next result, is an association of {\rm Theorem \ref{teo12}} and {\rm Theorem \ref{teo13}}.

\begin{corollary}\label{coro8} Let $\alpha\in(0,1]$, $f:[a,b]\rightarrow \mathbb{R}$ be an $n+1$ times $\alpha$-fractional differentiable function, with $r=1$, $s=\infty$, $t\in[a,b]$ and $\gamma, \beta,\delta,\rho\in\mathbb{C}$ such that $Re(\gamma)>0$, $Re(\beta)>0$, $Re(\delta)>0$, $Re(\rho)>0$ and $Re(\gamma)+p\geq q$. Then, the following inequality holds:
\begin{eqnarray*}
&&\left\vert \int_{x_{0}}^{t}\left\vert R_{n,f}\left( x_{0},\tau \right)
\right\vert \left\vert _{i}^{\rho }\mathcal{V}_{\gamma ,\beta ,\alpha }^{\delta
,p,q;n+1}f\left( \tau \right) \right\vert d_{\omega }\tau \right\vert  \\
&\leq &\left( \frac{\Gamma \left( \gamma +\beta \right) \left( \delta
\right) _{p}}{\Gamma \left( \beta \right) \left( \rho \right) _{q}}\right)
^{n}\frac{\left\vert t^{\alpha }-x_{0}^{\alpha }\right\vert ^{n+2}}{\alpha
^{n+2}\left( n+2\right) !}\left\Vert _{i}^{\rho }\mathcal{V}_{\gamma ,\beta ,\alpha
}^{\delta ,p,q;n+1}f\right\Vert _{\infty }^{2}.
\end{eqnarray*}
\end{corollary}

\begin{proof} Using {\rm Theorem \ref{teo12}} and {\rm Theorem \ref{teo13}}, we have
\begin{eqnarray*}
&&\left\vert \int_{x_{0}}^{t}\left\vert R_{n,f}\left( x_{0},\tau \right)
\right\vert \left\vert _{i}^{\rho }\mathcal{V}_{\gamma ,\beta ,\alpha }^{\delta
,p,q;n+1}f\left( \tau \right) \right\vert d_{\omega }\tau \right\vert  
\nonumber \\
&\leq &\frac{1}{2}\left( \frac{\Gamma \left( \gamma +\beta \right) \left(
\delta \right) _{p}}{\Gamma \left( \beta \right) \left( \rho \right) _{q}}%
\right) ^{n}\frac{\left( t^{\alpha }-x_{0}^{\alpha }\right) ^{n+2}}{\alpha
^{n+2}\left( n+2\right) !}\left\Vert _{i}^{\rho }\mathcal{V}_{\gamma ,\beta ,\alpha
}^{\delta ,p,q;n+1}f\right\Vert _{\infty ,\left[ x_{0},b\right] }^{2}+ 
\nonumber \\
&&\frac{1}{2}\left( \frac{\Gamma \left( \gamma +\beta \right) \left( \delta
\right) _{p}}{\Gamma \left( \beta \right) \left( \rho \right) _{q}}\right)
^{n}\frac{\left( x_{0}^{\alpha }-t^{\alpha }\right) ^{n+2}}{\alpha
^{n+2}\left( n+2\right) !}\left\Vert _{i}^{\rho }\mathcal{V}_{\gamma ,\beta ,\alpha
}^{\delta ,p,q;n+1}f\right\Vert _{\infty ,\left[ a,x_{0}\right] }^{2} 
\nonumber \\
&\leq &\left( \frac{\Gamma \left( \gamma +\beta \right) \left( \delta
\right) _{p}}{\Gamma \left( \beta \right) \left( \rho \right) _{q}}\right)
^{n}\frac{\left( t^{\alpha }-x_{0}^{\alpha }\right) ^{n+2}}{\alpha
^{n+2}\left( n+2\right) !}\left\Vert _{i}^{\rho }\mathcal{V}_{\gamma ,\beta ,\alpha
}^{\delta ,p,q;n+1}f\right\Vert _{\infty ,\left[ a,b\right] }^{2}+  \nonumber
\\
&&\left( \frac{\Gamma \left( \gamma +\beta \right) \left( \delta \right) _{p}%
}{\Gamma \left( \beta \right) \left( \rho \right) _{q}}\right) ^{n}\frac{%
\left( x_{0}^{\alpha }-t^{\alpha }\right) ^{n+2}}{\alpha ^{n+2}\left(
n+2\right) !}\left\Vert _{i}^{\rho }\mathcal{V}_{\gamma ,\beta ,\alpha }^{\delta
,p,q;n+1}f\right\Vert _{\infty ,\left[ a,b\right] }^{2}  \nonumber \\
&\leq &\left( \frac{\Gamma \left( \gamma +\beta \right) \left( \delta
\right) _{p}}{\Gamma \left( \beta \right) \left( \rho \right) _{q}}\right)
^{n}\frac{\left\vert t^{\alpha }-x_{0}^{\alpha }\right\vert ^{n+2}}{\alpha
^{n+2}\left( n+2\right) !}\left\Vert _{i}^{\rho }\mathcal{V}_{\gamma ,\beta ,\alpha
}^{\delta ,p,q;n+1}f\right\Vert _{\infty }^{2},
\end{eqnarray*}
which completes the proof.
\end{proof}


\section{Concluding remarks}

We introduced a new Taylor formula and the Taylor remainder via integral, through the truncated $\mathcal{V}$-fractional derivative and the $\mathcal{V}$-fractional integral. Besides that, we discussed the $\mathcal{V}$-fractional Holder's inequality and Cauchy-Schwartz inequality \cite{CAUCHY}, fundamental for the applications performed in section 5. The applications were restricted to inequalities from the truncated $\mathcal{V}$-fractional Taylor's remainder. Applications, such as approximations of functions by polynomials and an introduction to the $\mathcal{V}$-fractional Taylor's remainder by means of Lagrange’s form, will be presented in future works. In this context, one might think of extend the Taylor's formula, to the truncated $\mathcal{V}$-fractional in $\mathbb{R}^{n}$ \cite{JEC3}.

\bibliography{ref}
\bibliographystyle{plain}

\end{document}